\documentclass[11pt]{article}
\usepackage{amsmath,amsthm,amssymb,mathtools,mathrsfs,bbm}
\usepackage{mdframed}
\usepackage[hidelinks]{hyperref}
\usepackage{float}
\usepackage{algorithm2e}
\usepackage[noblocks]{authblk}
\usepackage{typearea}
\usepackage{natbib}
\usepackage{geometry}

\newgeometry{textwidth=16cm, textheight=25cm}

\numberwithin{equation}{section}

\newcommand{\PP}{\mathbb{P}}
\newcommand{\EE}{\mathbb{E}}
\newcommand{\RR}{\mathbb{R}}

\newcommand{\LL}{\mathbb{L}}
\newcommand{\MM}{\mathbb{M}}

\theoremstyle{definition}
\newtheorem{definition}{Definition}[section]

\newtheorem*{assumption*}{Assumption}

\newtheorem*{condition*}{Condition}

\theoremstyle{plain}
\newtheorem{theorem}[definition]{Theorem}
\newtheorem{proposition}[definition]{Proposition}
\newtheorem{lemma}[definition]{Lemma}

\theoremstyle{remark}

\title{Multifractional Stable Motion with Random Hurst Exponent}
\author{Fabian Mies \qquad Duuk Sikkens\thanks{Authors listed alphabetically\\ Correspondence: d.r.sikkens@tudelft.nl}\\
Delft University of Technology, The Netherlands}

\begin{document}

\maketitle

\begin{abstract}
    The fractional stable motion is a prototypical stochastic process exhibiting both heavy tails and long-range dependence, parameterized via a stability index $\alpha$ and a Hurst exponent $H$.
    We consider a nonstationary extension where the Hurst exponent is a function of time, and potentially random. 
    The construction admits the standard linear fractional stable motion as tangent process, and we exactly determine its local Hölder exponent in terms of the pointwise values of the Hurst function. This is in contrast to other definitions of multifractional processes, where the Hurst function needs to have additional regularity in time.

\end{abstract}

\section{Introduction}

This paper introduces and studies a new multifractional stable motion $(X(t))_{t \in \RR}$, which we call the Itô multifractional stable motion, given by
\begin{equation} \label{eq: definition ito multifractional stable motion}
    X(t) = \int_\RR(t-x)_+^{H(x)-\frac1\alpha} - (-x)_+^{H(x)-\frac1\alpha} \, dL(x), 
\end{equation}
where $(L(x))_{x \in \RR}$ is a symmetric $\alpha$-stable Lévy process with $\alpha\in(0,2]$, $(x)_+ = \max(x, 0)$, and $(H(x))_{x \in \RR}$ is an adapted process taking values between 0 and 1. The most prominent special case is the fractional Brownian motion $(W^H(t))_{t \in \RR}$, where $\alpha=2$, $H(x)=H\in(0,1)$ is constant, and \eqref{eq: definition ito multifractional stable motion} reduces to the Mandelbrot-van Ness representation \citep{mandelbrot_van_ness_fractional_brownian_motions_fractional_noises_and_applications}
$$
W^H(t) = \int_\RR (t-x)_+^{H-\frac12} - (-x)_+^{H-\frac12} \, dW(x),
$$
for a standard Brownian motion $W$. Due to its long-range dependence, the fractional Brownian motion has been used as a model in many applications, e.g.\ for network traffic, finance, autoregressive modeling and hydrology, see the overview in  \citet{doukhan_oppenheim_taqqu_theory_and_applications_of_long_range_dependence}.
However, the homogeneity of the Hölder exponent of the fractional Brownian motion makes it unfit for modeling non-stationary phenomena, as has been observed for financial time series \citep{frezza_bianchi_pianese_fractal_analysis_of_market_inefficiency_during_the_covid_19,bianchi_pantanella_pianese_modeling_stock_prices_by_multifractional_brownian_motion}, sun-spot data \citep{bibinger_cusum_tests_for_changes_in_the_hurst_exponent_and_volatility_of_fractional_brownian_motion}, or network traffic \citep{bianchi_henrique_vieira_ling_a_novel_network_traffic_predictor_based_on_multifractional_traffic_characteristic}. 
For this reason, \citet{peltier_levy_vehel_multifractional_brownian_motion_definition_and_preliminary_results} introduced the multifractional Brownian motion, where the constant $H$ is replaced by a function $H(t)$, as
\begin{align*}
    Y(t) = \int_\RR (t-x)_+^{H(t)-\frac12} - (-x)_+^{H(t)-\frac12} \, dW(x).
\end{align*}
They show that the pointwise Hölder exponent of this new process $t\mapsto Y(t)$ is $H(t)$ so that it changes along its path. 
However, there are two mathematical drawbacks to this construction. 
First, the definition in terms of a stochastic integral does not allow for a random multifractional parameter. This is only possible with additional regularity properties of $t\mapsto H(t)$ to allow for pathwise constructions of the stochastic integrals, see \citet{Ayache_taqqu_multifractional_processes_with_random_exponent}. 
Secondly, the Hölder regularity of the resulting process $t\mapsto X(t)$ depends on a Hölder condition on the multifractional parameter.
That is, if $t\mapsto H(t)$ is rough, then $Y(t)$ is also rough, regardless of the actual value of $H(t)$.

To allow for random Hurst functions, Ayache, Esser and Hamonier \citep{ayache_esser_hamonier_a_new_multifractional_process_with_random_exponent}, introduced a new multifractional Brownian motion given by  \eqref{eq: definition ito multifractional stable motion} with $\alpha=2$. 
Note that the kernels (indexed by $t$) are adapted to the natural filtration generated by the Brownian motion $W$, which means the resulting process may be understood as a collection of Itô integrals. 
To distinguish between the two types of multifractional Brownian motions, we refer to Peltier and Lévy Véhel's process $Y(t)$ as the classical multifractional Brownian motion, and to Ayache et al.'s process as the Itô multifractional Brownian motion. It has been shown by \citet{loboda_mies_steland_regularity_of_multifractional_moving_avarage_processes_with_random_hurst_exponent} that the pointwise Hölder exponent of the Itô multifractional Brownian motion at $t$ is equal to $H(t)$, irrespective of the Hölder regularity of $t \mapsto H(t)$; see also \citet{ayache_bouly_moving_average_multifractional_processes_with_random_exponent_lower_bounds_for_local_oscillations} for further path properties. 
Hence, the construction of Ayache, Esser and Hamonier addresses both limitations of the classical multifractional Brownian motion.

The heavy-tailed analogue of fractional Brownian motion is the linear fractional stable motion $(L^H(t))_{t \in \RR}$, introduced by \citet[Example 6.6]{taqqu_wolpert_infinite_variance_self_similar_processes_subordinate_to_a_poisson_measure}, see also \citet{samorodnitsky_taqqu_stable_non_gaussian_random_processes},
which is given by
$$
    L^H(t) = \int_\RR(t-x)_+^{H-\frac1\alpha} - (-x)_+^{H-\frac1\alpha} \, dL(x),
$$
where $(L(x))_{x \in \RR}$ is a symmetric $\alpha$-stable Lévy process with $\alpha \in (0,2)$. Stable distributions are a generalization of normal distributions in the sense that they serve as limit laws in a central limit theorem \citep{nolan_univariate_stable_distributions}. However, the tails of stable distributions decay polynomially. 
The process $L^H(t)$ is well-defined pointwise for any $t$ as a stochastic integral, see Section \ref{section: stochastic integration}. 
For $H> \frac{1}{\alpha}$, it admits a continuous version, but for $H<\frac{1}{\alpha}$, $L^H(t)$ looses a pathwise interpretation as the process is unbounded on any interval of positive length \citep[Ch.~10]{samorodnitsky_taqqu_stable_non_gaussian_random_processes}.
Thus, in the sequel, we restrict our attention to the case $H>\frac{1}{\alpha}$.
Mimicking the construction in the Gaussian case, a multifractional variant has been suggested \citep{stoev_taqqu_stochastic_properties_of_the_linear_multifractional_stable_motion}, where the fractional parameter $H$ is replaced by a deterministic function $H(t)$,
\begin{align}
    Y(t) = \int_\RR(t-x)_+^{H(t)-\frac1\alpha} - (-x)_+^{H(t)-\frac1\alpha} \, dL(x). \label{eqn: classical-mlfsm}
\end{align}
We refer to this process as the classical multifractional stable motion. The pointwise Hölder regularity is a bit more nuanced in the stable case. Indeed, \citet{ayache_hamonier_linear_multifractiona_stable_motion_fine_path_properties} find that, if $t \mapsto H(t)$ is sufficiently smooth, then the pointwise Hölder exponent of $Y(t)$ is $H(t)$, and the \emph{uniform} pointwise Hölder exponent is $H(t)-\frac1\alpha$.
See Section \ref{section: the ito multifractional stable motion} for the definition of these two slightly different Hölder exponents, which is not relevant for the multifractional Gaussian case, but only for the stable case. The same two mathematical drawbacks of the classical multifractional Brownian motion arise in the stable case: The formulation does not provide a natural framework for considering random multifractional parameters, and the Hölder regularity of the resulting process depends on the Hölder regularity of the multifractional parameter. 

In the Gaussian case, these drawbacks were resolved by considering an Itô multifractional Brownian motion, but no such attempts have been made in the stable regime. 
By introducing the new formulation \eqref{eq: definition ito multifractional stable motion} in the stable case, we also allow for random exponents. 
Moreover, the main result of this paper is that the pointwise Hölder exponent of $t \mapsto X(t)$ does not depend on the regularity of $t \mapsto H(t)$. 
This is illustrated by Figure \ref{fig:sim}, where sample paths of $X(t)$ and $Y(t)$ with the same Hurst function $H(t)$ are depicted. Here, the values of $H(t)$ are rather high, but the function $t \mapsto H(t)$ is very rough. 
As a consequence, the smoothness of $Y(t)$ is governed by the regularity of $t \mapsto H(t)$, while the smoothness of $X(t)$ is governed by the values of $H(t)$.
This finding is analogous to the result for the Gaussian case derived by \citet{loboda_mies_steland_regularity_of_multifractional_moving_avarage_processes_with_random_hurst_exponent}, but the mathematical study of stable processes requires a different analytical toolbox due to the heavy tails of the process, especially for the derivation of upper bounds on the Hölder exponents. Furthermore, we show that the tangent process of $X(t)$ at $t$ is the linear fractional stable motion with Hurst exponent $H(t)$, justifying its use as a nonstationary generalization of the latter.

\begin{figure}[H]
    \centering
    \includegraphics[width=\linewidth]{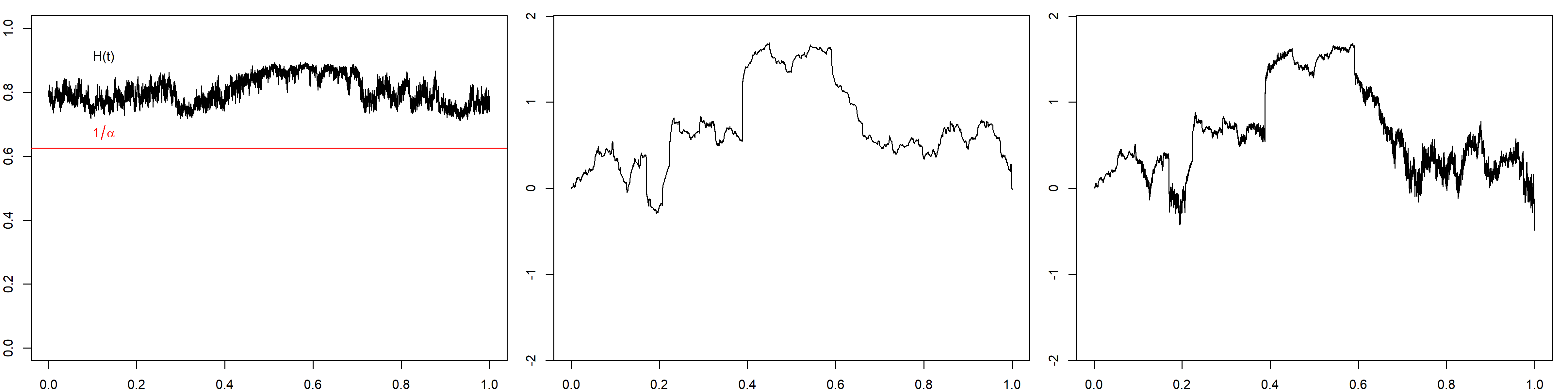}
    \caption{A rough Hurst function $H(t)$ with Hölder exponent $0.2$ (left), generated as $H(t)=0.8+0.1\tanh(W^{0.2}(t))$, and sample path of the corresponding classical multifractional stable motion $Y(t)$ with $\alpha=1.6$ (right) and the Itô multifractional stable motion $X(t)$ (center). The roughness of the $H(t)$ function does affect the path properties of $Y(t)$, whereas the properties of $X(t)$ only depend on the value of $H(t)$.}
    \label{fig:sim}
\end{figure}

The remainder of this paper is organized as follows. Section \ref{section: stochastic integration} will cover the Itô calculus against a symmetric $\alpha$-stable Lévy process. In Section \ref{section: the ito multifractional stable motion}, this Itô calculus is used to properly define the Itô multifractional stable motion, and we also treat the definition of Hölder exponents and compute the Hölder exponents of the Itô multifractional stable motion. Finally, all proofs will be collected in the appendix \ref{section: proofs}.

\paragraph{\textbf{Notation}}
Denote by $\LL^\alpha(\Omega)$ the Lebesgue space i.e.\ the space of random variables $X:\Omega \to \RR$ such that $\|X\|_{\LL^\alpha(\Omega)} = \left(\EE|X|^\alpha\right)^{1/\alpha} < \infty$. 
By $\Lambda^\alpha(\Omega)$ we denote the weak Lebesgue space, i.e.\ the space of random variables $X:\Omega \to \RR$ such that $\|X\|_{\Lambda^\alpha(\Omega)} = \left(\sup_{\lambda > 0} \lambda^\alpha \PP(|X| > \lambda)\right)^\frac1\alpha < \infty$.
Moreover, $\MM^\alpha(I)$ is the space of jointly measurable stochastic processes $F: \Omega \times I \to \RR$, adapted to a filtration $(\mathcal{F}_x)_{x \in I}$, such that $\|F\|_{\MM^\alpha(I)} = \left(\EE\int_I |F(x)|^\alpha \, dx\right)^{\frac1\alpha} < \infty$, and $\mathcal{S}(I)$ is the subspace of $\MM^\alpha(I)$ consisting of simple processes.

\section{Preliminaries on stable stochastic integrals} \label{section: stochastic integration}

This section briefly covers the Itô calculus against a symmetric $\alpha$-stable Lévy process $L$ developed by \citet{gine_marcus_the_central_limit_theorem_for_stochastic_integrals_with_respect_to_levy_processes} and extends this theory to an unbounded domain of integration.

Fix a  probability space $(\Omega, \mathcal{F}, \PP)$ and let $(L(x))_{x \in \RR}$ be a càdlàg symmetric $\alpha$-stable Lévy process with $\alpha \in (0, 2)$, standardized such that
$$
\EE\left[e^{i(L(x)-L(y))t}\right] = e^{-|x-y|\cdot |t|^\alpha}.
$$
Let $(\mathcal{F}_x)_{x \in \RR}$ be the natural filtration attributed to $L$. The goal is to develop stochastic integrals of the form $\int_\RR F \, dL$
for a suitable class of stochastic processes $(F(x))_{x \in \RR}$. For an interval $I \subseteq \RR$ we write $\MM^\alpha(I)$ for the space of jointly measurable stochastic processes $F:\Omega \times I \to \RR$, adapted to the filtration $(\mathcal{F}_x)_{x \in I}$, such that
$$
\|F\|_{\MM^\alpha(I)} = \left(\EE\int_I |F(x)|^\alpha \, dx\right)^{\frac1\alpha} < \infty. 
$$
The stochastic integral will be developed for $F \in \MM^\alpha(\RR)$, but first we restrict to the case where $I = (a,b]$ is a bounded interval and $F$ is simple, i.e. $F(\omega, x) = \sum_{k=1}^n \xi_k(\omega) \mathbbm{1}_{(x_{k-1}, x_k]}(x)$ for a partition $a = x_0 < \ldots < x_n = b$ and random variables $\xi_1, \ldots, \xi_n$ such that $\xi_k$ is $\mathcal{F}_{x_{k-1}}$-measurable. For simple processes with a decomposition as above the stochastic integral is (well-)defined in the usual way,
$$
\int_a^b F \, dL = \sum_{k=1}^n \xi_k \, (L(x_k) - L(x_{k-1})).
$$
Letting $\mathcal{S}((a,b])$ denote the space of simple processes, as a consequence of Equation (1.7) in \citet{gine_marcus_the_central_limit_theorem_for_stochastic_integrals_with_respect_to_levy_processes}, this definition yields a linear operator
\begin{align}
\begin{split} \label{eq: integral operator simple integrands}
\mathcal{S}((a,b]) &\to \Lambda^\alpha(\Omega) \\
F &\mapsto \int_a^b F \, dL
\end{split}
\end{align}
into the weak Lebesgue space $\Lambda^\alpha(\Omega)$ of random variables $X: \Omega \to \RR$ such that
\begin{equation} \label{eq: weak L alpha norm}
\|X\|_{\Lambda^\alpha(\Omega)} =  \left(\sup_{\lambda > 0} \lambda^\alpha \PP(|X| > \lambda)\right)^\frac1\alpha < \infty.
\end{equation}
The space $\Lambda^\alpha(\Omega)$ is a quasi-Banach space, which is a Banach space with a relaxed version of the triangle inequality $\|x + y\| \leq C(\|x\| + \|y\|)$ for some constant $C \geq 1$. For $\Lambda^\alpha(\Omega)$ this constant is $2 \vee 2^{1/\alpha}$ \citep[Equation (1.1.11)]{grafakos_classicall_fourier_analysis}. Lemma 3.3 in \citet{gine_marcus_the_central_limit_theorem_for_stochastic_integrals_with_respect_to_levy_processes} shows that the operator in \eqref{eq: integral operator simple integrands} is bounded for the special case $(a,b] = (0,1]$, but the proof of this lemma does not hinge on this special case and shows that the operators in \eqref{eq: integral operator simple integrands} are uniformly bounded, i.e. there is a constant $C(\alpha)$, independent of the interval $(a,b]$ and $F \in \mathcal{S}((a,b])$, such that
$$
\left\|\int_a^b F \, dL\right\|_{\Lambda^\alpha(\Omega)} \leq C(\alpha) \|F\|_{\MM^\alpha((a,b])}.
$$
Since $\mathcal{S}((a,b])$ is dense in $\MM^\alpha((a,b])$ \citep[Remark~3.2]{gine_marcus_the_central_limit_theorem_for_stochastic_integrals_with_respect_to_levy_processes} the integral operator \eqref{eq: integral operator simple integrands} uniquely extends to a uniformly bounded set of linear operators $\int_a^b \cdot \, dL: \MM^\alpha((a,b]) \to \Lambda^\alpha(\Omega)$.

Write $\mathcal{I}$ for the set of intervals $(a,b] \subseteq \RR$, partially ordered by set inclusion, and let $F \in \MM^\alpha(\RR)$. Due to uniform boundedness of the integral operators, it follows that the net $\left(\int_a^b F \, dL\right)_{(a,b] \in \mathcal{I}}$
is Cauchy in $\Lambda^\alpha(\Omega)$. Thus, we may define the stochastic integral $\int_\RR F \, dL$ as the limit of this net which results in a linear bounded operator
\begin{align*}
    \MM^\alpha(\RR) & \to \Lambda^\alpha(\Omega),\\
F &\mapsto \int_\RR F \, dL.
\end{align*}
Hence, the stochastic integral is well-defined for jointly measurable adapted integrands $F$ such that $\|F\|_{\MM^\alpha(\RR)}$ is finite and this stochastic integral defines the Itô multifractional stable motion \eqref{eq: definition ito multifractional stable motion}.

\section{The Itô Multifractional Stable Motion} \label{section: the ito multifractional stable motion}

Having established an Itô calculus against the symmetric $\alpha$-stable Lévy process, we may define the Itô multifractional stable motion for $\alpha \in (0,2)$ via \eqref{eq: definition ito multifractional stable motion}. To ensure integrability of the kernel, the (random) multifractional parameter $H(x)$ is required to be jointly measurable, we impose the following basic assumption.
\begin{quote}
    The stochastic process $H(x)$ is jointly measurable, adapted to $\mathcal{F}_x$, and bounded such that $\PP(\forall x \in \RR : \underline{H} \leq H(x) \leq \overline{H}) = 1$, for some $0<\underline{H} \leq \overline{H}<1$. 
\end{quote}
Under this restriction, it readily follows that the kernels $F_t(x,\omega)= (t-x)_+^{H(\omega, x)-1/\alpha} - (-x)_+^{H(\omega, x)-1/\alpha}$, indexed by $t \in \RR$, are in the class $\MM^\alpha(\RR)$, and thus the process $X(t) = \int_\RR F_t(x)\, dL(x)$ is well-defined.

In order to gain more insight into its path structure, it is instructive to consider the Lévy-Itô decomposition of the driving process as
$$
L(b) - L(a) = \int_a^b\int_{(-\gamma,\gamma)} y \, d\tilde N(x,y) + \int_a^b\int_{\RR\setminus(-\gamma,\gamma)} y \, dN(x,y) \quad a < b,
$$
where $N$ is a Poisson point process with intensity measure $dx\, d\mu(y)$, for the Lévy measure $d\mu(y) = \alpha |y|^{-\alpha-1} dy$, $\tilde{N}$ is the compensated Poisson process, and $\gamma > 0$. Since the Lévy measure is symmetric, the value of the truncation $\gamma$ does not matter.
Accordingly, under further assumptions on the Hurst function that will be provided later, for any $\gamma > 0$ and $t_0 \in \RR$, for $t > t_0$ the multifractional stable motion $X(t)$ may be decomposed as (see Lemma \ref{lemma: integral representation as integral against small and big jumps})
\begin{align*}
    X(t) 
    &=  \int_{-\infty}^{t_0} F_t(x) \, dL(x) + \int_{t_0}^t \int_{(-\gamma, \gamma)} y F_t(x)\, d\tilde{N}(x,y) + \int_{t_0}^t \int_{\RR \setminus (-\gamma, \gamma)} y F_t(x)\, dN(x,y).
\end{align*}
Letting $t_0 \to -\infty$ and $\gamma\to 0$, the first term tends to zero in $\Lambda^\alpha(\Omega)$ and the second term tends to zero in $\LL^2(\Omega)$. Thus, in $\Lambda^\alpha(\Omega)$ and for fixed $t$,
\begin{align}
    X(t)
    &= \lim_{\substack{t_0\to-\infty \\\gamma\to 0}} \int_{t_0}^t \int_{\RR \setminus (-\gamma, \gamma)} y F_t(x)\, d\tilde{N}(x,y)\nonumber\\
    &= \lim_{\substack{t_0\to-\infty \\\gamma\to 0}} \sum_{\substack{t_0 < x \leq t \\|\Delta L(x)|\geq\gamma}} \Delta L(x) \left((t-x)_+^{H(x)-\frac{1}{\alpha}} - (-x)_+^{H(x)-\frac1\alpha}\right),\label{eqn:jump-representation}
\end{align}
where $\Delta L(x)$ is the jump size at time $x$. 
We may hence think of the process $X(t)$ as a countable superposition of fractional power functions with random offsets and amplitudes. This representation can also be used for simulations with a deterministic Hurst function $H(x)$, or if $H(x)$ is independent of the driving Lévy motion $L(x)$, as follows.
Instead of discretizing the time domain, we fix a threshold $(t_0, \gamma)$ and approximate $X(t)$ in a given set of points $t = t_1 < \ldots < t_n$ by simulating the random atoms of the Poisson point process $N$ on $(t_0, t_n] \times \RR\setminus(-\gamma,\gamma)$. The total number of jumps $m = N((t_0, t_n] \times \RR \setminus (-\gamma, \gamma))$ follows a Poisson distribution with rate $\lambda = |(t_0, t_n]| \cdot \mu(\RR \setminus (-\gamma,\gamma)) = 2 (t_n-t_0) \gamma^{-\alpha}$. Conditioned on $m$, the distribution of the times and heights of the jumps are given by the normalized intensity measure of $N$. That is, the jump times are uniformly distributed on $(t_0, t_n]$, and the heights are distributed according to $\mu_{\restriction \RR \setminus (-\gamma, \gamma)} / \mu(\RR \setminus (-\gamma, \gamma))$, which has density equal to $\frac12\alpha\gamma^\alpha|y|^{-\alpha-1}$ supported on $\RR \setminus (-\gamma, \gamma)$. We see that the heights may be obtained as an independent product of a Rademacher and a $\text{Pareto}(\gamma, \alpha)$ distribution. Approximate sample paths may thus be generated as follows:
\begin{enumerate}
    \item Sample $m \sim \text{Pois}(2(t_n-t_0)\gamma^{-\alpha})$, the random number of jumps in the interval $(t_0,t_n]$ of size at least $\gamma$.
    \item Sample the jump times $x_1,\ldots, x_m \sim \text{Unif}(t_0,t_n]$.
    \item Sample the jump sizes $y_1, \ldots, y_m \sim \text{Rademacher} \times \text{Pareto}(\gamma, \alpha)$.
    \item For $t = t_1, \ldots, t_n$ compute $X(t) = \sum_{i=1}^m y_i \left((t-x_i)_+^{H(x_i)-1/\alpha} - (-x_i)_+^{H(x_i)-1/\alpha}\right)$.
\end{enumerate}
This procedure has been used to generate Figure \ref{fig:sim}, and Figure \ref{fig:sample-path-2} provides another illustration of the sample path using this procedure.

\begin{figure}[H]
    \centering
    \includegraphics[width=0.8\linewidth]{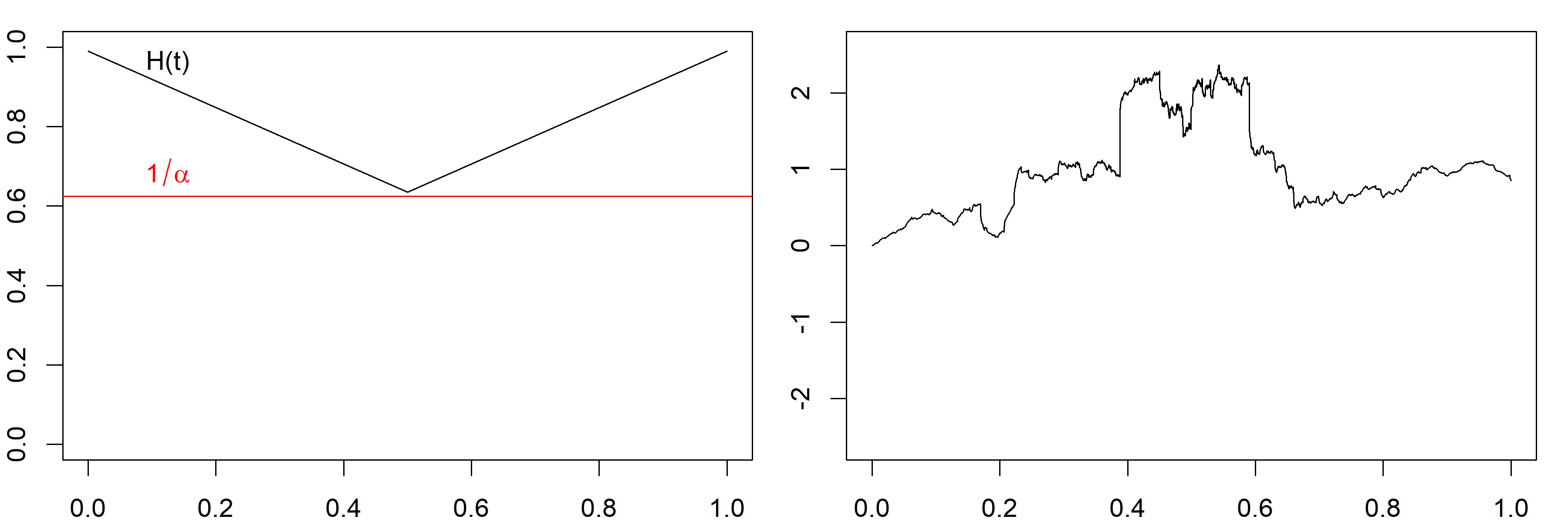}
    \caption{A Hurst function (left) with $\alpha = 1.6$ and $H(t) = 0.99 \wedge ((0.98-1/\alpha)|2t-1|+1/\alpha+0.01)$ and a sample path of its corresponding Itô multifractional stable motion (right).}
    \label{fig:sample-path-2}
\end{figure}

% min(1-delta, 2*(1-1/alpha-2*delta)*abs(t-0.5) + 1/alpha + delta)

The representation \eqref{eqn:jump-representation} also suggests that the local uniform Hölder exponent $\rho^\text{unif}_X(t)$ at time $t$ is $\rho^\text{unif}_X(t)=H(t)-\frac{1}{\alpha}$, corresponding to those powers. 
The main goal of this section, and the main result of this article, is to make this intuition about the regularity precise.

\subsection{Local regularity of sample paths}\label{sec:regularity}
Our main result regards local smoothness of the Itô multifractional stable motion. For a continuous function $f$, the local regularity may be expressed either in terms of the pointwise Hölder exponent $\rho_f(t)$ at a point $t$, via the uniform Hölder exponent $\rho_f^\text{unif}(I)$ over an interval $I$, or via the local uniform Hölder exponent $\rho_f^\text{unif}(t)$ at a point $t$.
These exponents are defined as

\begin{align*}
    \rho_f(t) &= \sup \left\{\rho \geq 0 : \limsup_{h \to 0} \frac{|f(t+h) - f(t)|}{|h|^\rho} = 0\right\}, \\
    \rho^\text{unif}_f(I) &= \sup\left\{\rho \geq 0 : \sup_{\substack{t_1, t_2 \in I \\ t_1 \neq t_2 }} \frac{|f(t_2) - f(t_1)|}{|t_2 - t_1|^\rho} < \infty\right\},\\
    \rho^{\text{unif}}_f(t) &= \sup \left\{\rho \geq 0 : \limsup_{h \downarrow 0} \,\, \sup_{\substack{t_1, t_2 \in [t-h, t+h] \\ t_1 \neq t_2}} \frac{|f(t_2) - f(t_1)|}{|t_2 - t_1|^\rho} = 0\right\}.
\end{align*}
Clearly, when $t$ is an interior point of $I$, these Hölder exponents are ordered as
\begin{align*}
    \rho_f^{\text{unif}}(I) \leq \rho_f^{\text{unif}}(t) \leq \rho_f(t).
\end{align*}
For the classical multifractional stable motion $Y(t)$, it is known that $\rho_Y^\text{unif}([a,b]) \geq \min_{t \in [a,b]} H(t) - \frac1\alpha$, under the assumption that $H$ takes values in a compact subset of $(1/\alpha, 1)$ and that $H$ is almost $\rho$-Hölder continuous for $\rho = \min_{t\in[a,b]}H(t)-\frac1\alpha$, see \cite[Cor.~5.3]{ayache_hamonier_linear_multifractiona_stable_motion_fine_path_properties}. 
We show that for the Itô multifractional stable motion $X(t)$, the same bound for the Hölder exponent holds but under much weaker regularity assumptions on $t \mapsto H(t)$. To be precise, we assume that the stochastic process $H$ admits a (deterministic) modulus of continuity $w:\RR_{\geq0} \to \RR_{\geq0}$, i.e.\ a continuous increasing function with $w(0)=0$, such that
\begin{equation} \label{eq: continuity assumption on H}
    \PP\left(\forall s, t \in \RR : |t-s| \leq 1 \implies |H(t) - H(s)| \leq w(|t-s|)\right) = 1.\tag{A}
\end{equation}
This restriction is weaker than a Hölder condition and allows, for example, for the case $w(h)=(1+|\log(h)|)^{-1}$. 
To derive the novel regularity result, we pursue a different mathematical approach compared to \citet{ayache_hamonier_linear_multifractiona_stable_motion_fine_path_properties}. 
Therein, the lower bound for the Hölder exponent is derived from a wavelet representation of $Y(t)$. In contrast, we employ a localized Kolmogorov-Chentsov criterion which has been developed previously to study the multifractional Gaussian case \citep{loboda_mies_steland_regularity_of_multifractional_moving_avarage_processes_with_random_hurst_exponent}. Their formulation allows for random and varying Hölder exponents and is provided in full in Section \ref{section: proofs}. The essential part of the proof is the following bound on the moments of scaled increments of $X(t)$, which requires different techniques than for the Gaussian process.

\begin{proposition}\label{proposition:increment-moments}
    Under condition \eqref{eq: continuity assumption on H}, for any $\delta>0$, there are constants $C > 0$ and $\epsilon > 0$ such that
    \begin{align*}
        \left\|\frac{X(t)-X(s)}{|t-s|^{\min_{r\in[s \wedge t,s \vee t]} H(r) - \frac{1}{\alpha} - \delta}}\right\|_{\Lambda^\alpha(\Omega)}^\alpha \leq C |t-s|, \quad \text{whenever } |t-s|\leq \epsilon.
    \end{align*}
    Hence, for any $\delta > 0$ and $\frac{\alpha}{1+\alpha\delta/2}<p<\alpha$, there are constants $C > 0$ and $\epsilon > 0$ such that
    \begin{align*}
        \left\|\frac{X(t)-X(s)}{|t-s|^{\min_{r\in[s \wedge t,s \vee t]} H(r) - \frac{1}{\alpha}-\delta}}\right\|_{\mathbb{L}^p(\Omega)}^p \leq C |t-s|,\quad \text{whenever } |t-s|\leq \epsilon.
    \end{align*}
\end{proposition}
To ensure positivity of the Hölder exponents, and thus continuous sample paths, we further impose that
\begin{equation} \label{eq: boundedness assumption on H}
    \underline{H} > \frac1\alpha, \tag{B},
\end{equation}
which is also presumed by \citet{ayache_hamonier_linear_multifractiona_stable_motion_fine_path_properties}.
Condition \eqref{eq: boundedness assumption on H} is necessary for continuity of the sample paths, even in the homogeneous case, see \cite[Ch.~10]{samorodnitsky_taqqu_stable_non_gaussian_random_processes}.
Under this assumption, and the continuity assumption \eqref{eq: continuity assumption on H}, we obtain a version of the Itô multifractional stable motion such that
\begin{equation} \label{eq: lower bound holder regularity}
\PP\left(\forall [a,b] \subseteq \RR : \rho^\text{\emph{unif}}_X([a,b]) \geq \min_{t \in [a,b]} H(t) - \frac1\alpha\right) = 1.
\end{equation}
In the sequel, we identify the Itô multifractional stable motion with its continuous modification.

To show that the lower bound on the Hölder exponent is sharp, we need a matching upper bound, i.e.\ a roughness result. 
For the classical multifractional stable motion, \citet[Thm.~6.1]{ayache_hamonier_linear_multifractiona_stable_motion_fine_path_properties} showed that $\rho^{\text{unif}}_Y([a,b]) \leq \min_{t \in [a,b]}H(t) - \frac1\alpha$ via a Wavelet representation, making extensive use of the smoothness of $t \mapsto H(t)$.
As we allow for rougher Hurst functions, we choose a different approach to derive the same lower bound by using a jump representation of $X(t)$ formalized in Lemma \ref{lemma: integral representation as integral against small and big jumps}. 

\begin{theorem}\label{theorem: equality holder exponents}
Assume \eqref{eq: continuity assumption on H} and \eqref{eq: boundedness assumption on H}. Then, for every $t_0 \in \RR$,
$$
\PP\left(\rho_X^\text{unif}(t_0) = H(t_0) - \frac1\alpha\right) = 1.
$$
Moreover, the uniform Hölder exponent is described uniformly over all intervals as
$$
\PP\left(\forall [a,b] \subseteq \RR : \rho^\text{unif}_X([a,b]) = \min_{t \in [a,b]} H(t)-\frac1\alpha\right) = 1.
$$
\end{theorem}

For the pointwise Hölder exponent, Theorem \ref{theorem: equality holder exponents} implies that $\rho_X(t_0)\geq H(t_0)-\frac{1}{\alpha}$ almost surely, which is weaker than the result of \citet{ayache_hamonier_linear_multifractiona_stable_motion_fine_path_properties} for the classical multifractional stable motion. However, we can obtain the following upper bound on the pointwise Hölder exponent, which is a direct consequence of Proposition \ref{proposition: tangent process ito multifractional stable motion} below.

\begin{theorem} \label{theorem: upper bound (non-uniform) poinwise holder exponent}
Assume \eqref{eq: continuity assumption on H} and \eqref{eq: boundedness assumption on H} and assume that ${w(h)\log(h) \to 0}$ as $h \downarrow 0$. Then, for all $t_0 \in \RR$, 
$$
\PP\left(\rho_X(t_0) \leq H(t_0)\right) = 1.
$$
\end{theorem}

Thus, we found that $H(t_0)-\frac{1}{\alpha} \leq \rho_X(t_0)\leq H(t_0)$ almost surely. For the stationary fractional stable motion $L^H(t)$, the gap between $\rho_{L^H}^\text{unif}(t) = H-1/\alpha$ and $\rho_{L^H}(t) = H$ is already present. Because the Itô multifractional stable motion is a nonstationary generalization of the linear fractional stable motion, we conjecture that indeed $\rho_X(t_0)=H(t_0)$, but we have been unable to establish this mathematically.

In comparison to the classical multifractional stable motion, our results remain valid under a  weaker assumption on the modulus of continuity of the Hurst function; where the Hurst function of the classical multifractional stable motion needed to be subject to Hölder regularity in order to obtain results on the regularity of the resulting process, the Hurst function of the Itô multifractional stable motion can admit any modulus of continuity. This was an expected improvement, since this was also the case for the Gaussian counterpart \citep{loboda_mies_steland_regularity_of_multifractional_moving_avarage_processes_with_random_hurst_exponent}. On the other hand, the wavelet techniques of Ayache and Hamonier allow the regularity of the resulting process to be determined to a finer degree. Where localized Kolmogorov-Chentsov criterion only captures the regularity of the resulting process up to the (local) Hölder exponent, with wavelet techniques it is also possible to discern the logarithmic terms in the modulus of continuity, see \citep[Theorems 5.1, 5.5, 6.1, 7.1]{ayache_hamonier_linear_multifractiona_stable_motion_fine_path_properties}.

\subsection{Tangent process}\label{sec:localization}
For the classical multifractional stable motion $Y(t)$, it is known that, for $t_0$ fixed, 
$$h^{-H(t_0)}\left(Y(t_0+hr)-Y(t_0)\right) \implies L^{H(t_0)}(r)$$ in terms of finite-dimensional distributions as $h \downarrow 0$, whenever $t \mapsto H(t)$ satisfies a local Hölder condition with exponent $H(t_0)$; see \citet[Thm.~5.1]{stoev_taqqu_stochastic_properties_of_the_linear_multifractional_stable_motion}.
Here $L^H(r)$ is the linear stable fractional motion with (deterministic) fractional parameter $H \in (0,1)$, provided in the introduction.
That is, the linear fractional stable motion is the tangent process of its multifractional extension. We find that the same is true for the Itô multifractional stable motion, under an even weaker condition on $H$. Moreover, we show that the convergence not only holds for the finite-dimensional distributions, but indeed weakly in the space of continuous functions \citep{billingsley_convergence_of_probability_measures}.

\begin{proposition} \label{proposition: tangent process ito multifractional stable motion} 
Assume \eqref{eq: continuity assumption on H} and \eqref{eq: boundedness assumption on H} and assume that ${w(h)\log(h) \to 0}$ as $h \downarrow 0$. Fix $t_0 \in \RR$ and $a < 0 < b$. Then, there is a symmetric $\alpha$-stable Lévy motion $(\tilde L(x))_{x \in \RR}$, independent of $H(t_0)$, such that, as $h \downarrow 0$,
$$
\left(\frac{X(t_0+hr)-X(t_0)}{h^{H(t_0)}}\right)_{r \in [a,b]} \implies \left(\int_\RR (r-x)_+^{H(t_0)-\frac1\alpha} - (-x)_+^{H(t_0)-\frac1\alpha} \, d\tilde L(x)\right)_{r \in [a,b]},
$$
weakly in the space $C([a,b])$.
\end{proposition}

While Theorem \ref{theorem: equality holder exponents} only requires continuity of the function $t \mapsto H(t)$, Proposition \ref{proposition: tangent process ito multifractional stable motion} imposes an upper bound on the modulus, i.e.\ an additional smoothness assumption. 
Nevertheless, the bound $w(h) \ll 1/|\log h|$ is very weak and satisfied, for example, by a Hölder function of any order.
We also would like to point out that the analogue of Theorem \ref{theorem: equality holder exponents} for the Brownian case \citep  {loboda_mies_steland_regularity_of_multifractional_moving_avarage_processes_with_random_hurst_exponent} does impose an upper bound on $w(h)$, as the proof of the roughness therein is based on the tangent process. In this sense, the exact local Hölder exponent of the multifractional stable motion can be obtained under even weaker conditions than for the Gaussian case.

\subsection{Open problems}

As explained above, it remains open to establish the exact value of the pointwise Hölder exponent of the Itô multifractional stable motion.
Furthermore, multivariate fractional Gaussian processes have been employed to improve accuracy of volatility forecasts \cite{bibinger_modeling_2026}, which motivates the study of multivariate extensions of the multifractional stable motion.

Another avenue for future research is statistical inference for the Hurst parameter $H(t)$, either parametrically or nonparametrically. 
The estimation of the linear fractional stable motion has been studied by \citep{Mazur2018, Ljungdahl2019, Ljungdahl2020, mies_estimation_2023}, estimation of the (classical) multifractional stable motion is studied by \citep{dangEstimationMultifractionalFunction2020, ayacheLinearMultifractionalStable2015, ayache_uniformly_2017}.
These latter three studies require that $t\mapsto H(t)$ is $\rho$-Hölder continuous, and $\rho>\sup_t H(t)$, which is the regime where the sample path regularity of the classical multifractional stable motion is described by the value of $H(t)$.
However, the rate of convergence established therein is a multiple of $\rho-\sup_{t}H(t)$, and thus vacuous in the rough case where $\rho< \sup_t H(t)$.
For the Gaussian case, it has recently been shown that the Itô-mBm allows for streamlined nonparametric estimation \citep{mies_rough_2025} compared to the classical mBm \citep{shen_hurst_2020}, achieving standard nonparametric rates of convergence without a lower bound on $\rho$
It can thus be conjectured that the same statistical simplifications emerge in the stable case.

\appendix
\section{Proofs} \label{section: proofs}

In this section, constants are denoted with a capital $C$. Unless specified otherwise, its dependencies are listed (e.g. $C(\epsilon), C(\alpha, \underline{H}, \overline{H})$). The value of constants may vary from line to line. To show \eqref{eq: lower bound holder regularity} we apply the localized Kolmogorov-Chentsov criterion developed in \citet{loboda_mies_steland_regularity_of_multifractional_moving_avarage_processes_with_random_hurst_exponent}. Below we state Theorem 2.1 from this article in dimension $d = 1$.

\begin{theorem} \label{theorem: refinement kolmogorov chentov theorem}
Let $(Y(t))_{t \in [0, T]}$ and $(a(t))_{t \in [0, T]}$ be stochastic processes taking values in $\RR$ and $(0, 1)$ respectively. Assume that $\inf_{t \in B} a(t)$ is measurable for any open or closed set $B \subseteq [0,T]$ and that $(a(t))_{t \in [0, T]}$ is lower semi-continuous and $\inf_{t \in [0,T]} a(t) > 0$. Suppose that there are $p > 0$, $\epsilon > 0$ and a constant $C > 0$ such that for all $s, t \in [0, T]$ with $|t - s| \leq \epsilon$,
$$
\EE\left|\frac{Y(t) - Y(s)}{|t-s|^{\min_{r \in [s \wedge t, s \vee t]} a(r)}}\right|^p \leq C |t-s|.
$$
Then $(Y(t))_{t \in [0, T]}$ has a modification $(\tilde Y(t))_{t \in [0, T]}$ such that
$$
\PP\left(\forall \gamma > 0 \,\, \forall [a,b] \subseteq [0, T] : \sup_{\substack{s, t \in [a, b] \\ s \neq t}} \frac{|\tilde Y(t) - \tilde Y(s)|}{|t-s|^{\min_{r \in [a,b]} a(r) - \gamma}} < \infty \right) = 1,
$$
as a consequence,
$$
\PP\left(\forall [a,b] \subseteq [0,T] : \rho^\text{unif}_{\tilde Y}([a,b]) \geq \min_{t \in [a,b]} a(t)\right) = 1.
$$
\end{theorem}

To bound the increments, we use the following lemma, which only concerns deterministic, real-valued integrals.

\begin{lemma} \label{lemma: bound on alpha norm difference kernel ito mbm}
Fix $\alpha \in (0, 2)$, $0 < \underline{H} < \overline{H} < 1$ and a modulus of continuity $w : \RR_{\geq0} \to \RR_{\geq 0}$. There exists a constant $C(\alpha, \underline{H}, \overline{H})$ such that for all $\epsilon \in (0, 1)$, all functions $H:\RR \to [\underline{H}, \overline{H}]$ admitting $w$ as a modulus of continuity and all $h \in (0, \frac12$),
\begin{align}
\int_{-\infty}^{t-\epsilon} \left|(t+h-x)^{H(x)-\frac1\alpha} - (t-x)^{H(x)-\frac1\alpha}\right|^\alpha \, dx &\leq C(\alpha, \underline{H}, \overline{H}) \, \epsilon^{\alpha(\underline{H}-1)} h^\alpha, \label{eq: lemma bounding difference kernel norms, part 1} \\
\int_{t-\epsilon}^t \left|(t+h-x)^{H(x)-\frac1\alpha} - (t-x)^{H(x)-\frac1\alpha}\right|^\alpha \, dx &\leq C(\alpha, \underline{H}, \overline{H}) \, h^{\alpha(H(t)-w(\epsilon))}, \label{eq: lemma bounding difference kernel norms, part 2} \\
\int_t^{t+h} (t+h-x)^{\alpha H(x) - 1} \, dx &\leq C(\alpha, \underline{H}, \overline{H}) \, h^{\alpha(H(t)-w(h))} \label{eq: lemma bounding difference kernel norms, part 3}.
\end{align}
\end{lemma}

\begin{proof}
Apply the mean value theorem to obtain $\xi_{t,h, x} \in [0, h]$ so that the left hand side of \eqref{eq: lemma bounding difference kernel norms, part 1} equals
\begin{align*} 
& h^\alpha \int_{-\infty}^{t-\epsilon}\left|H(x) - \frac1\alpha\right|^\alpha (t + \xi_{t,h,x} - x)^{\alpha (H(x) - 1) - 1} \, dx \\
\leq&\, C(\alpha, \underline{H}, \overline{H}) h^\alpha \int_{-\infty}^{t-\epsilon} (t - x)^{\alpha(\underline{H} - 1) - 1} \mathbbm{1}_{\{t + \xi_{t,h,x} - x \leq 1\}} + (t - x)^{\alpha(\overline{H} - 1) - 1} \mathbbm{1}_{\{t + \xi_{t,h,x} - x > 1\}} \, dx \\
\leq&\,  C(\alpha, \underline{H}, \overline{H}) h^\alpha \left(\int_{-\infty}^{t-\epsilon} (t-x)^{\alpha(\underline{H}-1)-1} \, dx + \int_{-\infty}^{t-\frac12} (t-x)^{\alpha(\overline{H}-1)-1} \, dx\right) \\
\leq&\,  C(\alpha, \underline{H}, \overline{H})\epsilon^{\alpha(\underline{H}-1)} \, h^\alpha. \\
\end{align*}
For the second term, substitute $x = t+h\hat x$ and use the fact that $H(t+h\hat x) \geq H(t) - w(\epsilon)$ whenever $t+h\hat x \in [t-\epsilon, t]$ to find that the left hand side of \eqref{eq: lemma bounding difference kernel norms, part 2} equals
\begin{align*}
& \int_{-\epsilon h^{-1}}^0 h^{\alpha H(t+h\hat x)} \left|(1-\hat x)^{H(t+h\hat x) - \frac1\alpha} - (-\hat x)^{H(t+h\hat x) - \frac1\alpha}\right|^\alpha \, d\hat x \\
\leq&\, h^{\alpha(H(t) - w(\epsilon))} \int_{-\infty}^0 \left|(1-\hat x)^{H(t+h\hat x) - \frac1\alpha} - (-\hat x)^{H(t+h\hat x) - \frac1\alpha}\right|^\alpha \, d\hat x \\
\leq&\, C(\alpha, \underline{H}, \overline{H}) h^{\alpha(H(t)-w(\epsilon))}.
\end{align*}
Finally,
$$
\int_t^{t+h} (t+h-x)^{\alpha H(x) - 1} \, dx \leq \int_t^{t+h} (t+h-x)^{\alpha \min_{x \in [t, t+h]} H(x)-1} \, dx \leq \frac1{\alpha \underline{H}} h^{\alpha(H(t) - w(h))}.
$$
\end{proof}

\begin{proof}[Proof of Proposition \ref{proposition:increment-moments}]
Let $\delta > 0$ be arbitrary.
Choose some $\epsilon \in (0,1)$. Assume that $s, t \in \RR$ are such that $0 < t-s \leq \epsilon$ and decompose
\begin{align*}
    X(t) - X(s) &= \underbrace{\int_{-\infty}^{s-\epsilon} (t-x)^{H(x) - \frac1\alpha} - (s-x)^{H(x) - \frac1\alpha} \, dL(x)}_{D_\epsilon}\\
    &+ \underbrace{\int_{s-\epsilon}^s (t-x)^{H(x) - \frac1\alpha} - (s-x)^{H(x) - \frac1\alpha} \, dL(x)}_{E_\epsilon}\\
    &+ \underbrace{\int_s^t (t-x)^{H(x) - \frac1\alpha} \, dL(x)}_{F}.\\
\end{align*}
From Lemma \ref{lemma: bound on alpha norm difference kernel ito mbm}, we find that $\epsilon$ affects the relevant exponents by at most $\omega(\epsilon)$. 
To make this effect sufficiently small, we choose $\epsilon>0$ such that so that $3w(\epsilon) \leq \delta$.
Then, the weak Lebesgue norm introduced in \eqref{eq: weak L alpha norm} is bounded as
\begin{align} \label{eq: decomposition alpha norm normalized difference ito mbm}
\begin{split}
&\left\|\frac{X(t)-X(s)}{|t-s|^{\min_{r \in [s, t]}H(r)}}\right\|^\alpha_{\Lambda^\alpha(\Omega)}\\
\leq& \, C(\alpha) \left(\left\|\frac{D_\epsilon}{|t-s|^{H(s)}}\right\|^\alpha_{\Lambda^\alpha(\Omega)} + \left\|\frac{E_\epsilon}{|t-s|^{H(s)}}\right\|^\alpha_{\Lambda^\alpha(\Omega)} + \left\|\frac{F}{|t-s|^{H(s)}}\right\|^\alpha_{\Lambda^\alpha(\Omega)}\right).
\end{split}
\end{align}
We bound the terms individually. For the first term, using \eqref{eq: lemma bounding difference kernel norms, part 1} of Lemma \ref{lemma: bound on alpha norm difference kernel ito mbm},
\begin{align} \label{eq: bound weak alpha moment, 1}
\begin{split}
\left\||t-s|^{-H(s)}D_\epsilon\right\|^\alpha_{\Lambda^\alpha(\Omega)} &\leq |t-s|^{-\alpha\overline{H}}\left\|D_\epsilon\right\|^\alpha_{\Lambda^\alpha(\Omega)} \\
&\leq C(\alpha) |t-s|^{-\alpha\overline{H}} \EE\int_{-\infty}^{s-\epsilon}\left|(t-x)^{H(x) - \frac1\alpha} - (s-x)^{H(x) - \frac1\alpha}\right|^\alpha \, dx\\
&\leq C(\alpha, \underline{H}, \overline{H}) \epsilon^{\alpha(\underline{H}-1)} |t-s|^{\alpha(1-\overline{H})}\\
&= C(\alpha, \underline{H}, \overline{H}, \epsilon)|t-s|^{\alpha(1-\overline{H})}.
\end{split}
\end{align}
To bound the second term, note that $H(s) \leq H(s-\epsilon) + w(\epsilon)$ almost surely. It follows that $\left\||t-s|^{-H(s)} E_\epsilon\right\|^\alpha_{\Lambda^\alpha(\Omega)} \leq |t-s|^{-\alpha w(\epsilon)} \left\||t-s|^{-H(s-\epsilon)} E_\epsilon\right\|^\alpha_{\Lambda^\alpha(\Omega)}$. Because $|t-s|^{-H(s-\epsilon)}$ is $\mathcal{F}_{s-\epsilon}$-measurable, we may bring this term inside the stochastic integral. By applying \eqref{eq: lemma bounding difference kernel norms, part 2} of Lemma \ref{lemma: bound on alpha norm difference kernel ito mbm} we find that
\begin{align} \label{eq: bound weak alpha moment, 2}
\begin{split}
    &\left\||t-s|^{-H(s)} E_\epsilon\right\|^\alpha_{\Lambda^\alpha(\Omega)} \\
    \leq& \, |t-s|^{-\alpha w(\epsilon)} \left\|\int_{s-\epsilon}^s |t-s|^{-H(s-\epsilon)} \left((t-x)^{H(x) - \frac1\alpha} - (s-x)^{H(x) - \frac1\alpha}\right) \, dL(x)\right\|^\alpha_{\Lambda^\alpha(\Omega)} \\
    \leq& \, C(\alpha) |t-s|^{-\alpha w(\epsilon)} \EE\left\{|t-s|^{-\alpha H(s-\epsilon)}\int_{s-\epsilon}^s  \left|(t-x)^{H(x) - \frac1\alpha} - (s-x)^{H(x) - \frac1\alpha}\right|^\alpha \, dx \right\}\\
    \leq& \, C(\alpha, \underline{H}, \overline{H}) |t-s|^{-\alpha w(\epsilon)}  \EE\left\{ |t-s|^{\alpha(H(s)- H(s-\epsilon) - w(\epsilon))} \right\} \\
    \leq& \, C(\alpha, \underline{H}, \overline{H})|t-s|^{-3\alpha w(\epsilon)}.
\end{split}
\end{align}
For the final term, $|t-s|^{-H(s)}$ is $\mathcal{F}_s$-measurable. So we can bring it inside of the integral without further complications. Using \eqref{eq: lemma bounding difference kernel norms, part 3} of Lemma \ref{lemma: bound on alpha norm difference kernel ito mbm},
\begin{align} \label{eq: bound weak alpha moment, 3}
\begin{split}
    \left\||t-s|^{-H(s)} F\right\|^\alpha_{\Lambda^\alpha(\Omega)} &= \left\|\int_s^t |t-s|^{-H(s)} (t-x)^{H(x)-\frac1\alpha} \, dL(x)\right\|^\alpha_{\Lambda^\alpha(\Omega)} \\
    &\leq C(\alpha) \, \EE\left\{|t-s|^{-\alpha H(s)} \int_s^t (t-x)^{\alpha H(x)-1} \, dx\right\} \\
    &\leq C(\alpha, \underline{H}, \overline{H}) \, |t-s|^{-\alpha w(|t-s|)} \\
    &\leq C(\alpha, \underline{H}, \overline{H}) |t-s|^{-\alpha w(\epsilon)}.
\end{split}
\end{align}
Plugging \eqref{eq: bound weak alpha moment, 1}, \eqref{eq: bound weak alpha moment, 2} and \eqref{eq: bound weak alpha moment, 3} into \eqref{eq: decomposition alpha norm normalized difference ito mbm} reveals that
\begin{align*}
\left\|\frac{X(t)-X(s)}{|t-s|^{\min_{r \in [s, t]}H(r)}}\right\|^\alpha_{\Lambda^\alpha(\Omega)} &\leq C(\alpha, \underline{H}, \overline{H}, \epsilon) \left(|t-s|^{\alpha(1-\overline{H})} + |t-s|^{-3\alpha w(\epsilon)} + |t-s|^{-\alpha w(\epsilon)}\right) \\
&\leq C(\alpha, \underline{H}, \overline{H}, \epsilon) |t-s|^{-3 \alpha w(\epsilon)}.
\end{align*}
Thus, for $\epsilon$ small enough such that $3w(\epsilon) \leq \delta$,
\begin{align*}
    \left\|\frac{X(t)-X(s)}{|t-s|^{\min_{r \in [s, t]}H(r) - \frac1\alpha - \delta}}\right\|^\alpha_{\Lambda^\alpha(\Omega)} &\leq C(\alpha, \underline{H}, \overline{H}, \epsilon) |t-s|^{1 + \alpha(\delta - 3w(\epsilon))} \leq C(\alpha, \underline{H}, \overline{H}, \epsilon) |t-s|.
\end{align*}
This proves the first statement. The second follows readily from the first and the continuous embedding $\Lambda^\alpha(\Omega) \hookrightarrow \mathbb{L}^p(\Omega)$ if $p < \alpha$. Indeed, for any $\delta>0$ and $\frac\alpha{1+\alpha\delta/2} < p < \alpha$,
    \begin{align*}
        \left\|\frac{X(t)-X(s)}{|t-s|^{\min_{r\in[s, t]} H(r) - \frac{1}{\alpha}-\delta}}\right\|_{\mathbb{L}^p(\Omega)}^p 
        &\leq C(\alpha, p) \left\|\frac{X(t)-X(s)}{|t-s|^{\min_{r\in[s, t]} H(r) - \frac{1}{\alpha} - \frac{\delta}{2}}}\right\|_{\Lambda^\alpha(\Omega)}^p |t-s|^{\frac{p\delta}{2}}\\
        &\leq C(\alpha, \underline{H}, \overline{H}, \epsilon, p) |t-s|^{\frac{p}{\alpha}+\frac{p\delta}{2}} \\
        &\leq C(\alpha, \underline{H}, \overline{H}, \epsilon, p) |t-s|.
    \end{align*}
\end{proof}

The proof for the roughness part of Theorem \ref{theorem: equality holder exponents} uses a decomposition of stochastic integrals $\int_a^b F(x) \, dL(x)$ against the Lévy motion $L$ into a sum of integrals against the small and big jumps of its associated Poisson point process $N$. Let $\gamma > 0$. For general integrands $F$ (not required to be adapted), the integral against big jumps is defined pathwise as the almost surely finite sum
$$
\int_a^b \int_{\RR \setminus (-\gamma, \gamma)} y F(x) \, dN(x,y) = \sum_{x \in (a,b]} \Delta L(x) F(x) \mathbbm{1}_{\RR \setminus (-\gamma, \gamma)}(\Delta L(x)),
$$
where $\Delta L(x) = L(x) - \lim_{y \uparrow x} L(y)$ is the jump process associated to $L$. For adapted left continuous square integrable integrands $F \in \MM^2((a,b])$, the integral $\int_a^b \int_{(-\gamma, \gamma)} y F(x) \, d\tilde N(x,y)$, with $\tilde N$ the compensated Poisson point process, is defined as the $\LL^2(\Omega)$ limit of integrals of simple processes. These integrals satisfy an Itô isometry,
$$
\EE\left|\int_a^b \int_{(-\gamma, \gamma)} y F(x) \, d\tilde N(x,y)\right|^2 = \EE\int_a^b\int_{(-\gamma,\gamma)} |yF(x)|^2 \, dx \, d\mu(y) = C(\alpha, \gamma) \EE\int_a^b |F(x)|^2 \, dx.
$$

\begin{lemma} \label{lemma: integral representation as integral against small and big jumps}
If $F \in \MM^2((a,b])$ is left continuous such that $\sup_{x\in(a,b]}|F(x)| \in \LL^2(\Omega)$, then
$$
\int_a^b F(x) \, dL(x) = \int_a^b \int_{(-\gamma,\gamma)} y F(x) \, d\tilde N(x, y) + \int_a^b \int_{\RR\setminus(-\gamma,\gamma)} y F(x) \, dN(x, y) \quad a.s.
$$
\end{lemma}

This result is of course expected to hold, but still non-trivial because the left-hand side is defined via the $\Lambda^\alpha(\Omega)$ norm inequality, and not via the standard stochastic integration theory against semimartingales.

\begin{proof}[Proof of Lemma \ref{lemma: integral representation as integral against small and big jumps}]
First let $F(x) = \sum_{k=1}^n \xi_k\mathbbm{1}_{(x_{k-1}, x_k]}(x)$ be simple. Then, by the Lévy-Itô decomposition,
\begin{align*}
     &\int_a^b \int_{(-\gamma,\gamma)} y F(x) \, d\tilde N(x, y) + \int_a^b \int_{\RR\setminus(-\gamma,\gamma)} y F(x) \, dN(x, y) \\
     =& \int_a^b \int_{(-\gamma,\gamma)} y \sum_{k=1}^n \xi_k \mathbbm{1}_{(x_{k-1}, x_k]}(x) \, d\tilde N(x, y) + \int_a^b \int_{\RR\setminus(-\gamma,\gamma)} y \sum_{k=1}^n \xi_k \mathbbm{1}_{(x_{k-1}, x_k]}(x) \, dN(x,y) \\
     =& \sum_{k=1}^n\xi_k \left(\int_{x_{k-1}}^{x_k} \int_{(-\gamma,\gamma)} y \, d\tilde N(x,y) + \int_{x_{k-1}}^{x_k} \int_{\RR\setminus(-\gamma,\gamma)} y \, dN(x,y)\right) \\
     =& \sum_{k=1}^n \xi_k (L(x_k) - L(x_{k-1})) \\
     =& \int_a^b F(x) \, dL(x).
\end{align*}
For general $F$, consider the sequence of simple processes given by $F_n(x) = F(\lfloor xn\rfloor /n)$ such that $\PP\left(\forall x \in (a,b] : F_n(x) \overset{n\to\infty}{\to} F(x)\right) = 1$. By the dominated convergence theorem, we also have $F_n\to F$ in $\LL^2(\Omega \times (a,b])$, hence also in $\LL^\alpha(\Omega \times (a,b])$ which implies
\begin{align*}
    \int_a^b F(x)\, dL(x) &= \lim_{n\to\infty} \int_a^b F_n(x)\, dL(x) \\
    &= \lim_{n\to\infty} \int_a^b \int_{(-\gamma,\gamma)} y F_n(x) \, d\tilde N(x, y) + \lim_{n\to\infty}  \int_a^b \int_{\RR\setminus(-\gamma,\gamma)} y F_n(x) \, dN(x, y).
\end{align*}
The first limit is $\int_a^b\int_{(-\gamma,\gamma)} y F(x) d\tilde{N}(x,y)$ as an $\LL^2(\Omega)$ limit, and the second is a limit of finite sums so it converges almost surely to $\int_a^b\int_{\RR\setminus(-\gamma,\gamma)} y F(x) \, dN(x,y)$.
\end{proof}

\begin{proof}[Proof of Theorem \ref{theorem: equality holder exponents}]
First we show \eqref{eq: lower bound holder regularity}, i.e.\ \underline{regularity of the paths}. For an integer $n > 0$, we consider the Itô multifractional stable motion $(X^n(t))_{t \in [-n, n]} = (X(t))_{t \in [-n, n]}$ restricted to $[-n, n]$. Let $0 < \delta < \underline{H} - 1/\alpha$ so that $a(t) = H(t)-1/\alpha-\delta, t \in [-n,n]$ satisfies the conditions of Theorem \ref{theorem: refinement kolmogorov chentov theorem}. From Theorem \ref{theorem: refinement kolmogorov chentov theorem} and Proposition \ref{proposition:increment-moments} we see that $(X^n(t))_{t \in [-n,n]}$ admits a modification $(\tilde X^n_\delta(t))_{t \in [-n,n]}$ such that
$$
\PP\left(\forall [a,b] \subseteq [-n, n] : \rho^\text{unif}_{\tilde X^n_\delta}([a,b]) \geq \min_{t \in [a,b]} H(t) - \frac1\alpha - \delta\right) = 1
$$
Since $\tilde X^n_\delta$ and $\tilde X^n_{\delta'}$ are continuous modifications for different $\delta$ and $\delta'$, they are indistinguishable and form a uniform modification $\tilde X^n$ such that, for any $0 < \delta < \underline{H} - 1/\alpha$,
$$
\PP\left(\forall [a,b] \subseteq [-n, n] : \rho^\text{unif}_{\tilde X^n}([a,b]) \geq \min_{t \in [a,b]} H(t) - \frac1\alpha - \delta\right) = 1
$$
Countably taking $\delta \downarrow 0$ and using that there are countably many integers $n > 0$ reveals 
\begin{equation} \label{eq: probability for omega* 1}
\PP\left(\forall n > 0, \forall [a,b] \subseteq [-n, n] : \rho^\text{unif}_{\tilde X^n}([a,b]) \geq \min_{t \in [a,b]} H(t) - \frac1\alpha\right) = 1.
\end{equation}
Moreover if $m \geq n$ then $(\tilde X^m(t))_{t \in [-n, n]}$ and $(\tilde X^n(t))_{t \in [-n, n]}$ are continuous modifications and therefore indistinguishable. Thus,
\begin{equation} \label{eq: probability for omega* 2}
\PP\left(\forall m \geq n, \forall t \in [-n, n] : \tilde X^m(t) = \tilde X^n(t)\right) = 1.
\end{equation}
Letting $\Omega^*$ be the intersection of the events in the probability in Equations \eqref{eq: probability for omega* 1} and \eqref{eq: probability for omega* 2}, for each $\omega \in \Omega^*$ and $t \in \RR$, take $n \geq |t|$ and set $\tilde X(\omega, t) = \tilde X^n(\omega, t)$, this value is independent of the chosen $n$ due to \eqref{eq: probability for omega* 2}. If $\omega \not\in \Omega^*$ set $\tilde X(\omega, t) = 0$. Then $(X(t))_{t \in \RR}$ and $(\tilde X(t))_{t \in \RR}$ are modifications. Let $\omega \in \Omega^*$, for an interval $[a,b] \subseteq \RR$ take $n > 0$ such that $[a,b] \subseteq [-n, n]$. By construction $\tilde X(\omega, t) = \tilde X^n(\omega, t)$ for $t \in [-n, n]$. Moreover, $\rho^\text{unif}_{\tilde X^n(\omega, \cdot)}([a,b]) \geq \min_{t \in [a,b]} H(\omega, t) - 1/\alpha$ by \eqref{eq: probability for omega* 1}. It follows that $\rho^\text{unif}_{\tilde X(\omega, \cdot)}([a,b]) \geq \min_{t \in [a,b]} H(\omega, t) - 1/\alpha$. So
$$
\PP\left(\forall [a,b] \subseteq \RR : \rho^\text{unif}_{\tilde X}([a,b]) \geq \min_{t \in [a,b]} H(t) - 1/\alpha\right) = 1.
$$
In the sequel, we identify the Itô multifractional stable motion $(X(t))_{t \in \RR}$ with its continuous modification. 

Next, we prove the lower bound on $\rho^\text{unif}_X$, i.e.\ \underline{roughness of the paths}. Let $t_0 \in \RR$, we show that $\PP(\rho_X^\text{unif}(t_0) \leq H(t_0)-\frac{1}{\alpha}) = 1$. Note that $\overline{H}-\underline{H} < 1/2$ because $\underline{H} > 1/\alpha > 1/2$. Let $0 < \rho < 1/2 - \overline{H} + \underline{H}$ and let $\delta > 0$ be such that $w(\delta) < \rho$. The assumptions \eqref{eq: continuity assumption on H} and \eqref{eq: boundedness assumption on H} ensure that for each $t \in \RR$, the kernel $F_t(\omega, x) = (t-x)_+^{H(\omega, x)-1/\alpha} - (-x)_+^{H(\omega, x)-1/\alpha}$ is continuous and square integrable. Thus, for $\gamma > 0$ and $t \in (t_0-\delta, t_0+\delta]$, using Lemma \ref{lemma: integral representation as integral against small and big jumps}, decompose
\begin{align*}
    X(t) &= \underbrace{\int_{-\infty}^{t_0-\delta} F_t(x) \, dL(x)}_{A(t)} + \underbrace{\int_{t_0-\delta}^{t_0+\delta} \int_{(-\gamma,\gamma)} y F_t(x) \, d\tilde N(x,y)}_{B(t)} \\
    &\qquad + \underbrace{\int_{t_0-\delta}^{t_0+\delta} \int_{\RR\setminus(-\gamma,\gamma)} y F_t(x) \, dN(x,y)}_{C(t)},
\end{align*}
where $A, B$ and $C$ have continuous sample paths. Let $\tau = \inf\{t \in (t_0-\delta, t_0+\delta] : N((t_0-\delta, t] \times \RR \setminus(-\gamma,\gamma)) > 0\}$ be the first arrival time in $(t_0-\delta, t_0+\delta]$ of a jump of size at least $\gamma$, with $\tau = t_0+\delta$ if no such jump occurs. We show that there is a sequence $h_n \downarrow 0$ such that
\begin{align}
    \PP\left(h_n^{-(H(t_0)-1/\alpha+\rho)}|A(\tau+h_n)-A(\tau)| \to 0\right) &= 1, \label{eq: for upperbound local holder exponent, A stays bounded} \\
    \PP\left(h_n^{-(H(t_0)-1/\alpha+\rho)}|B(\tau+h_n)-B(\tau)| \to 0\right) &= 1, \label{eq: for upperbound local holder exponent, B stays bounded} \\
    \PP\left(h_n^{-(H(t_0)-1/\alpha+\rho)}|C(\tau+h_n)-C(\tau)| \to \infty\right) &\geq \PP(N((t_0-\delta,t_0+\delta) \times \RR\setminus(-\gamma,\gamma)) > 0). \label{eq: for upperbound local holder exponent, C blows up}
\end{align}
First we show \eqref{eq: for upperbound local holder exponent, A stays bounded} and \eqref{eq: for upperbound local holder exponent, B stays bounded}. Write $\Sigma = \sigma(N(A) : A \subseteq (t_0-\delta, t_0+\delta] \times \RR \setminus (-\gamma, \gamma) \text{ measurable})$ and define the filtration $(\widehat{\mathcal{F}}_x)_{x \in \RR}$ by $\widehat{\mathcal{F}}_x = \sigma(\mathcal{F}_x \cup \Sigma)$. Since the Poisson point process is independently scattered, both $\sigma(L(y) - L(x) : x < y \leq t_0-\delta)$ and $\sigma(N(A) : A \subseteq (t_0-\delta, t_0+\delta] \times (-\gamma, \gamma) \text{ measurable})$ are independent of $\Sigma$ and it follows that $(L(x))_{x \leq t_0-\delta}$ is an $\alpha$-stable Lévy process such that $L(y) - L(x) \perp \widehat{\mathcal{F}}_x$ for $x < y \leq t_0-\delta$, and that $\tilde N_{\restriction (t_0-\delta, t_0+\delta] \times (-\gamma, \gamma)}$ is a martingale valued measure with respect to $(\widehat{\mathcal{F}}_x)_{x \in (t_0-\delta, t_0+\delta]}$. Since $\tau$ is $\Sigma$-measurable, the kernel $F_\tau$ is adapted to $(\widehat{\mathcal{F}}_x)_{x \in \RR}$ and the stochastic integrals $\int_{-\infty}^{t_0-\delta} F_\tau(x) \, dL(x)$ and $\int_{t_0-\delta}^{t_0+\delta} \int_{(-\gamma,\gamma)} y F_\tau(x) \, d\tilde N(x,y)$ exist. 

Now we show that $A(\tau) = \int_{-\infty}^{t_0-\delta} F_\tau(x) \, dL(x)$ and $B(\tau) = \int_{t_0-\delta}^{t_0+\delta} \int_{(-\gamma,\gamma)} y F_\tau(x) \, d\tilde N(x,y)$. Let $\mathcal{P}^n$ be a sequence of partitions of $(t_0-\delta, t_0+\delta]$ with $\text{mesh}(\mathcal{P}^n) \to 0$. Then, writing $A^n(\tau) = \sum_{t_j \in \mathcal{P}^n} A(t_j)\mathbbm{1}_{(t_j, t_{j+1}]}(\tau)$ and $F^n_\tau(x) = \sum_{t_j \in \mathcal{P}^n}F_{t_j}(x)\mathbbm{1}_{(t_j, t_{j+1}]}(\tau)$, we have the representation $A^n(\tau) = \int_{-\infty}^{t_0-\delta} F^n_\tau(x) \, dL(x)$. Moreover, since $A$ has continuous sample paths, $A^n(\tau) \to A(\tau)$ almost surely. We show that $\EE\int_\RR |F_\tau^n(x)-F_\tau(x)|^\alpha \, dx \to 0$ so that $A^n(\tau)$ converges to $\int_{-\infty}^{t_0-\delta} F_\tau(x) \, dL(x)$ in $\Lambda^\alpha(\Omega)$. Indeed,
\begin{align*}
\EE\int_\RR |F_\tau^n(x)-F_\tau(x)|^\alpha \, dx &= \sum_{t_j \in \mathcal{P}^n} \EE \mathbbm{1}_{(t_j, t_{j+1}]}(\tau) \int_\RR |F_{t_j}(x)-F_\tau(x)|^\alpha \, dx\\
&\leq \sum_{t_j \in \mathcal{P}^n} \EE \sup_{t \in (t_j, t_{j+1}]}\int_\RR |F_{t_j}(x)-F_t(x)|^\alpha \, dx.
\end{align*}
Let $\epsilon > 0$ be such that $w(\epsilon) < \underline{H}-1/\alpha$, then, by Lemma \ref{lemma: bound on alpha norm difference kernel ito mbm}, for $n$ big enough so that $\text{mesh}(\mathcal{P}^n) < \epsilon$ and for $t \in (t_j, t_{j+1}]$,
\begin{align*}
\int_\RR|F_{t_j}(x) - F_t(x)|^\alpha \, dx &\leq C(\alpha, \underline{H}, \overline{H}, \epsilon) (t-t_j)^{\alpha\left(\underline{H}-w(\epsilon)\right)}\\
&\leq C(\alpha, \underline{H}, \overline{H}, \epsilon) (t_{j+1}-t_j) \text{mesh}(\mathcal{P}_n)^{\alpha\left(\underline{H}-\frac1\alpha-w(\epsilon)\right)}.
\end{align*}
We conclude that $\EE\int_\RR|F_\tau^n(x)-F_\tau(x)|^\alpha \, dx \leq C(\alpha, \underline{H}, \overline{H}, \epsilon, \delta) \text{mesh}(\mathcal{P}^n)^{\alpha(\underline{H}-1/\alpha-w(\epsilon))} \to 0$ and that $A(\tau) = \int_{-\infty}^{t_0-\delta} F_\tau(x) \, dL(x)$. 

Now, in the same setting as Lemma \ref{lemma: bound on alpha norm difference kernel ito mbm} and using the same arguments, we find 
\begin{align} \label{eq: bounds for the L2 norm of the kernel}
\begin{split}
\int_{-\infty}^{t-\epsilon} \left|(t+h-x)^{H(x)-\frac1\alpha} - (t-x)^{H(x)-\frac1\alpha}\right|^2 \, dx &\leq C(\alpha, \underline{H}, \overline{H}, \epsilon) h^2, \\
\int_{t-\epsilon}^t \left|(t+h-x)^{H(x)-\frac1\alpha} - (t-x)^{H(x)-\frac1\alpha}\right|^2 \, dx &\leq C(\alpha, \underline{H}, \overline{H}) \, h^{2\left(H(t)-w(\epsilon)+\frac12-\frac1\alpha\right)}, \\
\int_t^{t+h} (t+h-x)^{2\left(H(x) - \frac1\alpha\right)} \, dx &\leq C(\alpha, \underline{H}, \overline{H}) \, h^{2\left(H(t)-w(h)+\frac12-\frac1\alpha\right)}.
\end{split}
\end{align}
This leads to the conclusion that \[\EE\int_\RR|F_\tau^n(x) - F_\tau(x)|^2 \, dx \leq C(\alpha, \underline{H}, \overline{H}, \epsilon, \delta) \text{mesh}(\mathcal{P}^n)^{2(\underline{H}-1/\alpha-w(\epsilon))} \to 0.\] 
Hence, by the Itô isometry and continuity of $B$, $B(\tau) = \int_{t_0-\delta}^{t_0+\delta} \int_{(-\gamma,\gamma)} y F_\tau(x) \, d\tilde N(x,y)$. 
We have thus shown that plugging the argument $\tau$ in $A(\tau)$ and $B(\tau)$ is identical to the respective stochastic integrals with kernels $F_\tau$, so they are subject to the $\Lambda^\alpha(\Omega)$ norm inequality and the Itô isometry respectively. By Lemma \ref{lemma: bound on alpha norm difference kernel ito mbm}, choosing $\epsilon > 0$ such that $w(\epsilon) < 1/\alpha + \underline{H}-\overline{H}-\rho$, for $|h| < \epsilon$,
\begin{align*}
    \left\|\frac{A(\tau+h)-A(\tau)}{h^{H(t_0)-\frac1\alpha+\rho}}\right\|_{\Lambda^\alpha(\Omega)}^\alpha &\leq C(\alpha) h^{\alpha\left(-\overline{H}+\frac1\alpha-\rho\right)} \EE\int_{-\infty}^{t_0-\delta}|F_{\tau+h}(x)-F_\tau(x)|^\alpha \, dx \\
    &\leq C(\alpha, \underline{H}, \overline{H}, \epsilon) h^{\alpha\left(\frac1\alpha+\underline{H}-\overline{H}-\rho-w(\epsilon)\right)} \overset{h\to0}{\to} 0.
\end{align*}
Similarly, using \eqref{eq: bounds for the L2 norm of the kernel} and choosing $\epsilon > 0$ such that $w(\epsilon) < 1/2-\overline{H}+\underline{H}-\rho$, for $|h| < \epsilon$,
\begin{align*}
\left\|\frac{B(\tau+h)-B(\tau)}{h^{H(t_0)-\frac1\alpha+\rho}}\right\|_{\LL^2(\Omega)}^2 &\leq C(\alpha, \gamma) h^{2\left(-\overline{H}+\frac1\alpha-\rho\right)} \EE\int_\RR |F_{\tau+h}(x) - F_\tau(x)|^2 \, dx\\
&\leq C(\alpha, \underline{H}, \overline{H}, \epsilon, \gamma) h^{2\left(\frac12-\overline{H}+\underline{H}-w(\epsilon)-\rho\right)} \overset{h\to0}{\to} 0.
\end{align*}
We conclude that $h^{-(H(t_0)-1/\alpha+\rho)}(A(\tau+h)-A(\tau)) \to 0$ and $h^{-(H(t_0)-1/\alpha+\rho)}(B(\tau+h)-B(\tau)) \to 0$ in probability, \eqref{eq: for upperbound local holder exponent, A stays bounded} and \eqref{eq: for upperbound local holder exponent, B stays bounded} follow.

We now show \eqref{eq: for upperbound local holder exponent, C blows up}.
Suppose that $N((t_0-\delta, t_0+\delta) \times \RR\setminus(-\gamma,\gamma)) > 0$. Then $\tau \in (t_0-\delta, t_0+\delta)$ and $|\Delta L(\tau)| \geq \gamma$. Suppose furthermore that for $\epsilon > 0$ we have $N((\tau, \tau+\epsilon) \times \RR\setminus(-\gamma,\gamma)) = 0$ so that $\tau$ is the only jump of size at least $\gamma$ in $(t_0-\delta, \tau+\epsilon)$, then for $0 < h < \epsilon$,
$$
\frac{|C(\tau+h)-C(\tau)|}{h^{H(t_0)-\frac1\alpha+\rho}} = h^{H(\tau)-H(t_0)-\rho} |\Delta L(\tau)| \geq \gamma h^{w(\delta)-\rho} \to \infty.
$$
We conclude that
\begin{align*}
&\PP\left(\frac{|C(\tau+h)-C(\tau)|}{h^{H(t_0)-\frac1\alpha+\rho}} \overset{h\downarrow0}\to \infty\right)\\
\geq& \, \PP(N((t_0-\delta, t_0+\delta) \times \RR\setminus(-\gamma,\gamma)) > 0 \wedge N((\tau, \tau+\epsilon) \times \RR\setminus(-\gamma,\gamma)) = 0)\\
\overset{\epsilon\downarrow0}{\to}& \, \PP(N((t_0-\delta, t_0+\delta) \times \RR\setminus(-\gamma,\gamma)) > 0).
\end{align*}
Combining \eqref{eq: for upperbound local holder exponent, A stays bounded}, \eqref{eq: for upperbound local holder exponent, B stays bounded} and \eqref{eq: for upperbound local holder exponent, C blows up} reveals that for each $0 < \rho < 1/2 - \overline{H}+\underline{H}$ there is a $\delta > 0$ such that, for all $\gamma > 0$,
$$
\PP\left(\sup_{\substack{s, t \in [t_0-\delta,t_0+\delta] \\ s \neq t}} \frac{|X(t)-X(s)|}{|t-s|^{H(t_0)-1/\alpha+\rho}} = \infty\right) \geq \PP(N((t_0-\delta, t_0+\delta) \times \RR \setminus (-\gamma,\gamma)) > 0) \overset{\gamma\downarrow0}{\to} 1.
$$
Countably taking $\rho \downarrow 0$ and corresponding $\delta \downarrow 0$ shows $\PP(\rho^\text{unif}_X(t_0) \leq H(t_0)-1/\alpha) = 1$.

By \eqref{eq: lower bound holder regularity} and continuity of $H$ it follows that $\PP(\rho^\text{unif}_X(t_0) = H(t_0)-1/\alpha) = 1$. Thus, $\PP(\forall t_0 \in \mathbb{Q} : \rho^\text{unif}_X(t_0) = H(t_0) - 1/\alpha) = 1$. In this event, we have, for any $[a,b] \subseteq \RR$,
$$
\rho^\text{unif}_X([a,b]) \leq \inf_{t \in (a,b) \cap \mathbb{Q}} \rho^\text{unif}_X(t) = \inf_{t \in (a,b) \cap \mathbb{Q}} H(t)-\frac1\alpha = \min_{t \in [a,b]} H(t) - \frac1\alpha.
$$
We conclude that $\PP\left(\forall [a,b] \subseteq \RR : \rho^\text{unif}_X([a,b]) = \min_{t \in [a,b]} H(t)-\frac1\alpha\right) = 1$.

\end{proof}

For the localizability result Proposition \ref{proposition: tangent process ito multifractional stable motion}, we need the following technical lemma.

\begin{lemma} \label{lemma: technical lemma to change exponent}
Let $\alpha \in (0, 2)$ and $0 < \underline{c} < \overline{c} < 1$. Then there is a constant $C(\alpha, \underline{c}, \overline{c})$ such that for all functions $a: \RR \to [\underline{c}, \overline{c}]$ and $b: \RR \to [\underline{c}, \overline{c}]$ satisfying $|a(x) - b(x)| \leq \Delta$ for all $x \in \RR$, and for all $h \in (0,1/e)$,
\begin{align*}
&\int_\RR \left|\left((h-x)_+^{a(x)-\frac1\alpha} - (-x)_+^{a(x)-\frac1\alpha}\right) - \left((h-x)_+^{b(x)-\frac1\alpha} - (-x)_+^{b(x)-\frac1\alpha}\right)\right|^\alpha \, dx \\
\leq& \, C(\alpha, \underline{c}, \overline{c}) \, \Delta^\alpha h^{\alpha \, \underline{a \wedge b}} \, |\log h|^\alpha,
\end{align*}
where $\underline{a \wedge b} = \inf_{x \in \RR} (a(x) \wedge b(x))$.
\end{lemma}

\begin{proof}
Split up the integral into
\begin{align}
&\int_{-\infty}^0 \left|\left((h-x)^{a(x)-\frac1\alpha} - (-x)^{a(x)-\frac1\alpha}\right) - \left((h-x)^{b(x)-\frac1\alpha} - (-x)^{b(x)-\frac1\alpha}\right)\right|^\alpha \, dx,  \label{eq: first part for the technical lemma}\\
+& \int_0^h \left|(h-x)^{a(x)-\frac1\alpha} - (h-x)^{b(x)-\frac1\alpha}\right|^\alpha \, dx. \label{eq: second part for the technical lemma}
\end{align}
To bound \eqref{eq: first part for the technical lemma}, substitute $x = h \hat x$ and apply the mean value theorem to obtain $\xi_{\hat x, h}$ between $a(h\hat x)$ and $b(h\hat x)$ such that
\begin{align*}
&\int_{-\infty}^0 \left|\left((h-x)^{a(x)-\frac1\alpha} - (-x)^{a(x)-\frac1\alpha}\right) - \left((h-x)^{b(x)-\frac1\alpha} - (-x)^{b(x)-\frac1\alpha}\right)\right|^\alpha \, dx\\
=&\int_{-\infty}^0 \left|h^{a(h\hat x)}\left((1-\hat x)^{a(h \hat x) - \frac1\alpha} - (-\hat x)^{a(h \hat x) - \frac1\alpha}\right) - h^{b(h\hat x)}\left((1-\hat x)^{b(h \hat x) - \frac1\alpha} - (-\hat x)^{b(h \hat x) - \frac1\alpha}\right)\right|^\alpha \, d\hat x \\
=& \int_{-\infty}^0 |b(h\hat x) - a(h\hat x)|^\alpha h^{\alpha \xi_{\hat x, h}}\left|\log h\left((1-\hat x)^{\xi_{\hat x, h}-\frac1\alpha} - (-\hat x)^{\xi_{\hat x, h}-\frac1\alpha} \right)\right. \\
&\left.+ \left((1-\hat x)^{\xi_{\hat x, h}-\frac1\alpha} \log(1-\hat x) - (-\hat x)^{\xi_{\hat x, h} - \frac1\alpha}\log(-\hat x)\right)\right|^\alpha \, d\hat x \\
\leq& \, 2^\alpha \Delta^\alpha h^{\alpha \, \underline{a \wedge b}}\, |\log h|^\alpha \left[\int_{-\infty}^0 \left|(1-\hat x)^{\xi_{\hat x, h}-\frac1\alpha} - (-\hat x)^{\xi_{\hat x, h}-\frac1\alpha}\right|^\alpha \, d\hat x \right. \\
&\left. \hspace*{2.9cm} + \int_{-\infty}^0\left|(1-\hat x)^{\xi_{\hat x, h}-\frac1\alpha} \log(1-\hat x) - (-\hat x)^{\xi_{\hat x, h} - \frac1\alpha}\log(-\hat x)\right|^\alpha \, d\hat x\right].
\end{align*}
Both of the integrals between the square brackets are bounded above, independently from $h$ and the functions $a$ and $b$. To bound \eqref{eq: second part for the technical lemma}, the mean value theorem implies that
\begin{align*}
\int_0^h \left|(h-x)^{a(x)-\frac1\alpha} - (h-x)^{b(x)-\frac1\alpha}\right|^\alpha \, dx &\leq \Delta^\alpha \int_0^h (h-x)^{\alpha \underline{a \wedge b} - 1} |\log(h-x)|^\alpha \, dx \\
&\hspace*{-0.4cm}\overset{x = h-h\hat x}{\leq} \Delta^\alpha h^{\alpha \underline{a \wedge b}} \int_0^1 \hat x^{\alpha \underline{c}-1} |\log \hat x + \log h|^\alpha \, d \hat x \\
&\leq 2^\alpha \Delta^\alpha h^{\alpha \underline{a \wedge b}} |\log h|^\alpha \int_0^1 \hat x^{\alpha \underline{c}-1} (|\log \hat x |^\alpha + 1) \, d\hat x.
\end{align*}
\end{proof}

\begin{proof}[Proof of Proposition \ref{proposition: tangent process ito multifractional stable motion}]
To show finite-dimensional convergence, let $r \in [a,b]$ be fixed. First it will be established that, for $q \in (0,1)$ small enough,
\begin{align} \label{eq: replace H(x) by H(t) for localizability}
\begin{split}
& \frac{X(t_0+hr)-X(t_0)}{h^{H(t_0)}} \\
=& \, h^{-H(t_0-h^q)} \int_{t_0-h^q}^\infty (t_0+hr-x)_+^{H(t_0-h^q) - \frac1\alpha} - (t_0-x)_+^{H(t_0-h^q) - \frac1\alpha} \, dL(x) + o_\PP(1).
\end{split}
\end{align}
The first step to establishing \eqref{eq: replace H(x) by H(t) for localizability} is to show that
\begin{equation} \label{eq: localizability random multifractional parameter, step 1}
h^{-H(t_0)} \int_{-\infty}^{t_0-h^q} (t_0+hr-x)^{H(x) - \frac1\alpha} - (t_0-x)^{H(x)-\frac1\alpha} \, dL(x) \to 0 \hspace*{0.5cm} \text{in probability as $h \downarrow 0$.}
\end{equation}
Assume that $r > 0$. Then, using $\epsilon = h^q$ in \eqref{eq: lemma bounding difference kernel norms, part 1} from Lemma \ref{lemma: bound on alpha norm difference kernel ito mbm} shows that
\begin{align*}
&\left\|h^{-H(t_0)} \int_{-\infty}^{t_0-h^q} (t_0+hr-x)^{H(x) - \frac1\alpha} - (t_0-x)^{H(x)-\frac1\alpha} \, dL(x)\right\|_{\Lambda^\alpha(\Omega)}^\alpha\\
\leq& \, C(\alpha) h^{-\alpha\overline{H}} \EE\int_{-\infty}^{t_0-h^q}\left|(t_0+hr-x)^{H(x)-\frac1\alpha} - (t_0-x)^{H(x)-\frac1\alpha}\right|^\alpha \, dx \\
\leq& \, C(\alpha, \underline{H}, \overline{H})h^{-\alpha \overline{H}} h^{\alpha q(\underline{H}-1)} (hr) \\ 
\leq& \, C(\alpha, \underline{H}, \overline{H}, r) h^{\alpha \left[1-\overline{H} - q(1-\underline{H})\right]} \overset{h\downarrow0}{\to} 0 
\end{align*}
whenever $q < \frac{1-\overline{H}}{1-\underline{H}}$. If $r < 0$ then we take $\epsilon = hr + h^q$ (under the assumption that $h$ is close enough to zero so that this is positive) to conclude that
\begin{align*}
&\left\|h^{-H(t_0)} \int_{-\infty}^{t_0-h^q} (t_0+hr-x)^{H(x) - \frac1\alpha} - (t_0-x)^{H(x)-\frac1\alpha} \, dL(x)\right\|_{\Lambda^\alpha(\Omega)}^\alpha\\
\leq& \, C(\alpha) h^{-\alpha \overline{H}} \EE\int_{-\infty}^{t_0+hr - \epsilon}\left|(t_0-x)^{H(x)-\frac1\alpha} - (t_0+hr-x)^{H(x)-\frac1\alpha}\right|^\alpha \, dx \\
\leq& \, C(\alpha, \underline{H}, \overline{H}) h^{-\alpha \overline{H}} (hr+h^q)^{\alpha(\underline{H}-1)}(h|r|)^\alpha \\
\leq& \, C(\alpha, \underline{H}, \overline{H}, r) h^{\alpha\left[1-\overline{H} - q(1-\underline{H})\right]} (1+rh^{1-q})^{\alpha(\underline{H}-1)} \overset{h\downarrow0}{\to} 0
\end{align*}
under the assumption that $q < \frac{1-\overline{H}}{1-\underline{H}}$.
Since convergence in a weak Lebesgue space implies convergence in probability, \eqref{eq: localizability random multifractional parameter, step 1} follows. Thus
$$
\frac{X(t_0+hr) - X(t_0)}{h^{H(t_0)}} = h^{-H(t_0)} \int_{t_0-h^q}^\infty (t_0+hr-x)_+^{H(x)-\frac1\alpha} - (t_0-x)_+^{H(x)-\frac1\alpha} \, dL(x) + o_\PP(1).
$$
The second step is replacing $H(t_0)$ in the prefactor $h^{-H(t_0)}$ by $H(t_0-h^q)$, making it $\mathcal{F}_{t_0-h^q}$-measurable. By the mean value theorem, with probability one,
$$
\left|h^{-H(t_0)} - h^{-H(t_0-h^q)}\right| \leq h^{-H(t_0-h^q)} w(h^q) |\log h| h^{-w(h^q)}.
$$
Thus, using that $h^{-H(t_0-h^q)}$ is $\mathcal{F}_{t_0-h^q}$-measurable and using \eqref{eq: lemma bounding difference kernel norms, part 2} and \eqref{eq: lemma bounding difference kernel norms, part 3} in Lemma \ref{lemma: bound on alpha norm difference kernel ito mbm}, we have 
\begin{align*}
&\left\|\left(h^{-H(t_0)} - h^{-H(t_0-h^q)}\right) \int_{t_0-h^q}^\infty  (t_0+hr-x)_+^{H(x)-\frac1\alpha} - (t_0-x)_+^{H(x)-\frac1\alpha} \, dL(x)\right\|_{\Lambda^\alpha(\Omega)}^\alpha \\
\leq & \, w(h^q)^\alpha |\log h|^\alpha h^{-\alpha w(h^q)} \left\|h^{-H(t_0-h^q)} \int_{t_0-h^q}^\infty  (t_0+hr-x)_+^{H(x)-\frac1\alpha} - (t_0-x)_+^{H(x)-\frac1\alpha} \, dL(x)\right\|_{\Lambda^\alpha(\Omega)}^\alpha \\
\leq& \, C(\alpha) w(h^q)^\alpha |\log h|^\alpha h^{-\alpha w(h^q)}\EE h^{-\alpha H(t_0-h^q)} \int_{t_0-h^q}^\infty \left|(t_0+hr-x)_+^{H(x)-\frac1\alpha} - (t_0-x)_+^{H(x)-\frac1\alpha}\right|^\alpha \, dx \\
\leq& \, C(\alpha, \underline{H}, \overline{H}) w(h^q)^\alpha |\log h|^\alpha h^{-\alpha w(h^q)} \EE h^{-\alpha H(t_0-h^q)} \\
& \cdot \left((h|r|)^{\alpha (H(t_0 \wedge (t_0+hr)) - w(h^q))} + (h|r|)^{\alpha(H(t_0 \wedge (t_0+hr))-w(h|r|))}\right) \\
\leq& \, C(\alpha, \underline{H}, \overline{H}, r) w(h^q)^\alpha |\log h|^\alpha h^{-2\alpha w(h^q) - \alpha w(h^q + h|r|)} \overset{h\downarrow0}{\to} 0.
\end{align*}
We conclude that
$$
\frac{X(t_0+hr) - X(t_0)}{h^{H(t_0)}} = h^{-H(t_0-h^q)} \int_{t_0-h^q}^\infty  (t_0+hr-x)_+^{H(x)-\frac1\alpha} - (t_0-x)_+^{H(x)-\frac1\alpha} \, dL(x) + o_\PP(1).
$$
The third and final step is approximating $H(x)$ in the integrand by $H(t_0-h^q)$. This will be done using Lemma \ref{lemma: technical lemma to change exponent}, which shows that, with probability one,
\begin{align*}
    &h^{-\alpha H(t_0-h^q)}\int_{t_0-h^q}^{(t_0+hr) \vee t_0} \left|\left((t_0+hr-x)_+^{H(x)-\frac1\alpha} - (t_0-x)_+^{H(x)-\frac1\alpha}\right) - \right.\\
    &\left.\hspace*{4.03cm}\left((t_0+hr-x)_+^{H(t_0-h^q)-\frac1\alpha} - (t_0-x)_+^{H(t_0-h^q)-\frac1\alpha}\right)\right|^\alpha \, dx \\
    =& \, h^{-\alpha H(t_0-h^q)} \int_{-h^q}^{hr \vee 0} \left|\left((hr-x)_+^{H(x+t_0)-\frac1\alpha} - (-x)_+^{H(x+t_0)-\frac1\alpha}\right) - \right.\\
    &\left.\hspace*{3.25cm}\left((hr-x)_+^{H(t_0-h^q)-\frac1\alpha} - (-x)_+^{H(t_0-h^q)-\frac1\alpha}\right)\right|^\alpha \, dx \\
    \leq& \, C(\alpha, \underline{H}, \overline{H}) h^{-\alpha H(t_0-h^q)} \left(\sup_{x \in [t_0-h^q, (t_0+hr) \vee t_0]} |H(x) - H(t_0-h^q)|\right)^\alpha \\
    & \cdot (h|r|)^{\alpha(H(t_0-h^q) - w(h^q + h|r|))} |\log(h|r|)|^\alpha \\
    \leq& \, C(\alpha, \underline{H}, \overline{H}, r) w(h^q + h|r|)^\alpha |\log h|^\alpha h^{-\alpha w(h^q + h|r|)}.
\end{align*}
Taking expectation, and using that $h^{-H(t_0-h^q)}$ is $\mathcal{F}_{t_0-h^q}$-measurable so that we can take it inside the stochastic integral, we find that
\begin{align*}
&\left\|h^{-H(t_0-h^q)}\left(\int_{t_0-h^q}^\infty (t_0+hr-x)_+^{H(x)-\frac1\alpha} - (t_0-x)_+^{H(x)-\frac1\alpha} \, dL(x) -\right.\right.\\
&\left.\left.\hspace*{2cm} \int_{t_0-h^q}^\infty (t_0+hr-x)_+^{H(t_0-h^q)-\frac1\alpha} - (t_0-x)_+^{H(t_0-h^q)-\frac1\alpha} \, dL(x)\right)\right\|_{\Lambda^\alpha(\Omega)}^\alpha \\
\leq&\, C(\alpha, \underline{H}, \overline{H}, r) w(h^q + h|r|)^\alpha |\log h|^\alpha h^{-\alpha w(h^q + h|r|)} \overset{h\downarrow0}{\to} 0.
\end{align*}
It follows that
$$
\frac{X(t_0+hr) - X(t_0)}{h^{H(t_0)}} = h^{-H(t_0-h^q)} \int_{t_0-h^q}^\infty  (t_0+hr-x)_+^{H(t_0-h^q)-\frac1\alpha} - (t_0-x)_+^{H(t_0-h^q)-\frac1\alpha} \, dL(x) + o_\PP(1),
$$
so that \eqref{eq: replace H(x) by H(t) for localizability} holds. Now, the integrand on the right hand side is $\mathcal{F}_{t_0-h^q}$-measurable, so it is independent of $((L(x)-L(t_0-h^q))_{x \geq t_0-h^q}$ and we may introduce a symmetric $\alpha$-stable Lévy motion $(\tilde L(x))_{x \in \RR}$, independent of the processes $L$ and $H$, such that
\begin{align*}
& \left(\int_{t_0-h^q}^\infty  (t_0+hr-x)_+^{H(t_0-h^q)-\frac1\alpha} - (t_0-x)_+^{H(t_0-h^q)-\frac1\alpha} \, dL(x)\right)_{r \in [a,b]} \\ 
\overset{d}{=}& \left(\int_{t_0-h^q}^\infty  (t_0+hr-x)_+^{H(t_0-h^q)-\frac1\alpha} - (t_0-x)_+^{H(t_0-h^q)-\frac1\alpha} \, d\tilde L(x)\right)_{r \in [a,b]}.
\end{align*}
By applying the same steps as before but in reverse, and using that $(\tilde L(x))_{x \in \RR}$ is $\frac1\alpha$-self-similar and has stationary increments, we find, for multiple values $r_1, \ldots, r_n$,
\begin{align*}
& \left(\frac{X(t_0+hr_k)-X(t_0)}{h^{H(t_0)}}\right)_{k=1 \ldots n}\\
=& \left(h^{-H(t_0-h^q)} \int_{t_0-h^q}^\infty (t_0+hr_k-x)_+^{H(t_0-h^q)-\frac1\alpha} - (t_0-x)_+^{H(t_0-h^q)-\frac1\alpha} \, dL(x)\right)_{k=1 \ldots n} + o_\PP(1)\\
\overset{d}{=}& \left(h^{-H(t_0-h^q)} \int_{t_0-h^q}^\infty (t_0+hr_k-x)_+^{H(t_0-h^q)-\frac1\alpha} - (t_0-x)_+^{H(t_0-h^q)-\frac1\alpha} \, d\tilde L(x)\right)_{k=1 \ldots n} + o_\PP(1)\\
=& \left(h^{-H(t_0)} \int_\RR (t_0+hr_k-x)_+^{H(t_0)-\frac1\alpha} - (t_0-x)_+^{H(t_0)-\frac1\alpha} \, d\tilde L(x)\right)_{k=1 \ldots n} + o_\PP(1) \\
\overset{d}{=}& \left(\int_\RR (r_k - x)_+^{H(t_0) - \frac1\alpha} - (-x)_+^{H(t_0) - \frac1\alpha} \, d\tilde L(x)\right)_{k=1 \ldots n} +o_\PP(1).
\end{align*}
This shows convergence in finite-dimensional distribution. To show that the convergence is functional, let $\tilde X^{h, t_0}(r) = h^{-H(t_0)}(X(t_0+hr)-X(t_0))$ be the rescaled difference process. We will show that the collection of processes $(\tilde X^{h, t_0})_{h>0}$ is tight in $C([a,b])$ by applying Theorem 23.7 from \citet{kallenberg_foundations_of_modern_probability}. Of course $\tilde X^{h,t_0}(0) = 0$ for each $h > 0$, so $(\tilde X^{h, t_0}(0))_{h > 0}$ is tight. Let $p < \alpha$ and let $r, s \in [a,b]$ and suppose $0 \leq r-s \leq 1$, then
\begin{align*}
& \left\|\frac{X(t_0+hr)-X(t_0+hs)}{h^{H(t_0)}}\right\|^p_{\LL^p(\Omega)} \\
\leq& \, h^{-pw(h(b-a))} |r-s|^{p \underline{H}} \left\|\frac{X(t_0+hr)-X(t_0+hs)}{|hr-hs|^{H(t_0+hs)}}\right\|_{\LL^p(\Omega)}^p \\
\leq& \, C(\alpha, p) h^{-pw(h(b-a))} |r-s|^{p \underline{H}} \left(\left\|\frac{X(t_0+hr)-X(t_0+hs)}{|hr-hs|^{H(t_0+hs)}}\right\|_{\Lambda^\alpha(\Omega)}^\alpha \right)^{\frac p\alpha}.
\end{align*}
Applying \eqref{eq: bound weak alpha moment, 1}, \eqref{eq: bound weak alpha moment, 2} and \eqref{eq: bound weak alpha moment, 3} with $\epsilon = h^q$ for $q \in (0,1)$ reveals that
\begin{align*}
&\left\|\frac{X(t_0+hr)-X(t_0+hs)}{|hr-hs|^{H(t_0+hs)}}\right\|_{\Lambda^\alpha(\Omega)}^\alpha \\
\leq& \, C(\alpha, \underline{H}, \overline{H}) \left[h^{q\alpha(\underline{H}-1)}|hr-hs|^{\alpha (1-\overline{H})} + |hr-hs|^{-3\alpha w(h^q)} + |hr-hs|^{-\alpha w(h(b-a))}\right].
\end{align*}
Now $h^{\alpha(1-\overline{H}-q(1-\underline{H}))} \to 0$ if $q < \frac{1-\overline{H}}{1-\underline{H}}$. Moreover, $h^{-3\alpha w(h^q)} \to 1$ and $h^{-\alpha w(h(b-a))} \to 1$, so for all $h \in (0, \delta)$ in a small enough neighborhood of $0$, we find
$$
\left\|\frac{X(t_0+hr)-X(t_0+hs)}{|hr-hs|^{H(t_0+hs)}}\right\|_{\Lambda^\alpha(\Omega)}^\alpha \leq C(\alpha, \underline{H}, \overline{H}, \delta) |r-s|^{-3\alpha w(\delta^q)}.
$$
Thus, for all $h \in (0,\delta)$,
$$
\left\|\frac{X(t_0+hr)-X(t_0+hs)}{h^{H(t_0)}}\right\|^p_{\LL^p(\Omega)} \leq C(\alpha, \underline{H}, \overline{H}, p, \delta) |r-s|^{p\underline{H} - 3p w(\delta^q)}.
$$
Choosing $1/\underline{H} < p < \alpha$ and $\delta > 0$ small enough so that $3pw(\delta^q) < p\underline{H} - 1$, we find that the moment criteria of Theorem 23.7 in \citet{kallenberg_foundations_of_modern_probability} is met, so $(\tilde X^{h,t_0})_{h > 0}$ is tight. Functional convergence thus follows from Prokhorov's theorem.
\end{proof}

\begin{proof}[Proof of Theorem \ref{theorem: upper bound (non-uniform) poinwise holder exponent}]
By Proposition \ref{proposition: tangent process ito multifractional stable motion}, as $h \downarrow 0$ we have $h^{-H(t_0)}(X(t_0+h)-X(t_0)) \to \int_\RR (1-x)_+^{H(t_0)-1/\alpha} - (-x)_+^{H(t_0)-1/\alpha} \, d\tilde L(x)$ in distribution. So for any $\rho > 0$, as $h \downarrow 0$, $h^{-H(t_0)-\rho}(X(t_0+h)-X(t_0)) \to \infty$ in distribution and thus also in probability. So there is a sequence $h_n \downarrow 0$ such that the divergence is almost sure and we find $\rho_X(t_0) \leq H(t_0)$ almost surely.
\end{proof}

\bibliographystyle{apalike}
\bibliography{references}

@article{bibinger_modeling_2026,
	title = {Modeling and {Forecasting} {Realized} {Volatility} with {Multivariate} {Fractional} {Brownian} {Motion}},
	issn = {0735-0015, 1537-2707},
	url = {https://www.tandfonline.com/doi/full/10.1080/07350015.2026.2696490},
	doi = {10.1080/07350015.2026.2696490},
	language = {en},
	urldate = {2026-07-06},
	journal = {Journal of Business \& Economic Statistics},
	author = {Bibinger, Markus and Yu, Jun and Zhang, Chen},
	month = jun,
	year = {2026},
	pages = {1--27},
}

@article{shen_hurst_2020,
	title = {Hurst function estimation},
	volume = {48},
	issn = {0090-5364},
	url = {https://projecteuclid.org/journals/annals-of-statistics/volume-48/issue-2/Hurst-function-estimation/10.1214/19-AOS1825.full},
	doi = {10/gn8zkk},
	number = {2},
	urldate = {2022-01-25},
	journal = {The Annals of Statistics},
	author = {Shen, Jinqi and Hsing, Tailen},
	month = apr,
	year = {2020},
	pages = {838--862},
}

@misc{mies_rough_2025,
	title = {Rough {Hurst} function estimation},
	url = {http://arxiv.org/abs/2511.10103},
	doi = {10.48550/arXiv.2511.10103},
	urldate = {2026-07-06},
	publisher = {arXiv},
	author = {Mies, Fabian and Wilkens, Benedikt},
	month = nov,
	year = {2025},
	note = {arXiv:2511.10103 [math.ST]},
}

@article{ayache_uniformly_2017,
	title = {Uniformly and strongly consistent estimation for the {Hurst} function of a {Linear} {Multifractional} {Stable} {Motion}},
	volume = {23},
	issn = {1350-7265},
	url = {https://projecteuclid.org/journals/bernoulli/volume-23/issue-2/Uniformly-and-strongly-consistent-estimation-for-the-Hurst-function-of/10.3150/15-BEJ781.full},
	doi = {10.3150/15-BEJ781},
	number = {2},
	urldate = {2026-07-06},
	journal = {Bernoulli},
	author = {Ayache, Antoine and Hamonier, Julien},
	month = may,
	year = {2017}
	
}

@article{Mazur2018,
	title = {Estimation of the linear fractional stable motion},
	volume = {26},
	url = {https://projecteuclid.org/journals/bernoulli/volume-26/issue-1/Estimation-of-the-linear-fractional-stable-motion/10.3150/19-BEJ1124.short},
	doi = {10.3150/19-BEJ1124},
	number = {1},
	journal = {Bernoulli},
	author = {Mazur, Stepan and Otryakhin, Dmitry and Podolskij, Mark},
	year = {2020},
	pages = {226--252},
}

@article{Ljungdahl2020,
	title = {Multi-{Dimensional} {Parameter} {Estimation} of {Heavy}-{Tailed} {Moving} {Averages}},
	volume = {49},
	doi = {10.1111/sjos.12527},
	journal = {Scandinavian Journal of Statistics},
	author = {Ljungdahl, Mathias Mørck and Podolskij, Mark},
	year = {2021},
	pages = {593--624},
}

@article{Ljungdahl2019,
	title = {A {Minimal} {Contrast} {Estimator} for the {Linear} {Fractional} {Stable} {Motion}},
	volume = {23},
	url = {http://arxiv.org/abs/1911.04341},
	doi = {10.1007/s11203-020-09216-2},
	journal = {Statistical Inference for Stochastic Processes},
	author = {Ljungdahl, Mathias Mørck and Podolskij, Mark},
	year = {2020},
	pages = {381--413},
}

@article{mies_estimation_2023,
	title = {Estimation of mixed fractional stable processes using high-frequency data},
	volume = {51},
	url = {http://arxiv.org/abs/2208.07453},
	doi = {10.1214/23-AOS2312},
	number = {5},
	urldate = {2022-09-07},
	journal = {Annals of Statistics},
	author = {Mies, Fabian and Podolskij, Mark},
	year = {2023},
	pages = {1946--1964},
}

@article{ayache_bouly_moving_average_multifractional_processes_with_random_exponent_lower_bounds_for_local_oscillations,
TITLE = {Moving average Multifractional Processes with Random Exponent: Lower bounds for local oscillations},
VOLUME = {146},
DOI = {10.1016/j.spa.2022.01.003},
JOURNAL = {Stochastic Processes Appl.},
AUTHOR = {A. Ayache and F. Bouly},
YEAR = {2021},
PAGES = {143--163},
}

@inproceedings{bianchi_henrique_vieira_ling_a_novel_network_traffic_predictor_based_on_multifractional_traffic_characteristic,
TITLE = {A novel network traffic predictor based on multifractal traffic characteristic},
VOLUME = {2},
ISBN = {0-7803-8794-5},
DOI = {10.1109/GLOCOM.2004.1378048},
BOOKTITLE = {{IEEE} {Global} {Telecommunications} {Conference}, 2004. {GLOBECOM} '04.},
PUBLISHER = {IEEE},
AUTHOR = {G. Bianchi and F. Henrique and T. Vieira and L. L. Ling},
YEAR = {2004},
PAGES = {680--684},
}

@article{bibinger_cusum_tests_for_changes_in_the_hurst_exponent_and_volatility_of_fractional_brownian_motion,
TITLE = {Cusum Tests for Changes in the {{Hurst}} Exponent and Volatility of Fractional {{Brownian}} Motion},
AUTHOR = {M. Bibinger},
YEAR = {2020},
JOURNAL = {Stat. Probab. Lett.},
VOLUME = {161},
PAGES = {108725},
DOI = {10.1016/j.spl.2020.108725},
ISBN = {1904.04556v1},
}

@article{bianchi_pantanella_pianese_modeling_stock_prices_by_multifractional_brownian_motion,
TITLE = {Modeling stock prices by multifractional {Brownian} motion: an improved estimation of the pointwise regularity},
VOLUME = {13},
ISSN = {1469-7688, 1469-7696},
DOI = {10.1080/14697688.2011.594080},
NUMBER = {8},
JOURNAL = {Quantatative Finance},
AUTHOR = {S. Bianchi and A. Pantanella and A. Pianese},
YEAR = {2013},
PAGES = {1317--1330},
}

@article{frezza_bianchi_pianese_fractal_analysis_of_market_inefficiency_during_the_covid_19,
TITLE = {Fractal analysis of market (in)efficiency during the {COVID}-19},
VOLUME = {38},
ISSN = {15446123},
DOI = {10.1016/j.frl.2020.101851},
JOURNAL = {Finance Research Letters},
AUTHOR = {M. Frezza and S. Bianchi and A. Pianese},
YEAR = {2021},
PAGES = {101851},
}

@book{samorodnitsky_taqqu_stable_non_gaussian_random_processes,
AUTHOR = {G. Samorodnitsky and M. S. Taqqu},
TITLE = {Stable Non-Gaussian Random Processes},
SUBTITLE = {Stochastic Models with Infinite Variane},
YEAR = 1994,
PUBLISHER = {Chapman and Hall},
ISBN = {0-412-05171-0},
ADDRESS = {New York}
}

@book{nolan_univariate_stable_distributions,
AUTHOR = {J. P. Nolan},
TITLE = {Univariate Stable Distributions},
YEAR = 2020,
PUBLISHER = {Springer},
ADDRESS = {Cham},
ISBN = {978-3-030-52914-7},
DOI = {10.1007/978-3-030-52915-4}
}

@article{stoev_taqqu_stochastic_properties_of_the_linear_multifractional_stable_motion,
AUTHOR = {S. Stoev and M. S. Taqqu},
TITLE = {Stochastic Properties of the Linear Multifractional Stable Motion},
JOURNAL = {Adv. Appl. Probab.},
VOLUME = {36},
NUMBER = {4},
PAGES = {1085--1115},
YEAR = {2004},
DOI = {10.1239/aap/1103662959}
}

@article{ayache_hamonier_linear_multifractiona_stable_motion_fine_path_properties,
AUTHOR = {A. Ayache and J. Hamonier},
TITLE = {Linear multifractional stable motion: fine path properties},
JOURNAL = {Rev. Mat. Iberoam.},
VOLUME = {30},
NUMBER = {4},
PAGES = {1301-1354},
YEAR = {2014},
DOI = {10.4171/RMI/816}
}

@article{loboda_mies_steland_regularity_of_multifractional_moving_avarage_processes_with_random_hurst_exponent,
AUTHOR = {D. Loboda and F. Mies and A. Steland},
TITLE = {Regularity of multifractional moving average processes with random Hurst exponent},
JOURNAL = {Stochastic Processes Appl.},
VOLUME = {140},
PAGES = {21-48},
YEAR = {2021},
DOI = {10.1016/j.spa.2021.05.008}
}

@article{gine_marcus_the_central_limit_theorem_for_stochastic_integrals_with_respect_to_levy_processes,
AUTHOR = {E. Giné and M. B. Marcus},
TITLE = {The Central Limit Theorem for Stochastic Integrals with Respect to Levy Processes},
JOURNAL = {Ann. Probab.},
VOLUME = {11},
NUMBER = {1},
PAGES = {58 -- 77},
YEAR = {1983},
DOI = {10.1214/aop/1176993660},
}

@article{mandelbrot_van_ness_fractional_brownian_motions_fractional_noises_and_applications,
AUTHOR = {Mandelbrot, B. B. and van Ness, J. W.},
TITLE = {Fractional Brownian Motions, Fractional Noises and Applications},
JOURNAL = {SIAM Rev.},
VOLUME = {10},
NUMBER = {4},
PAGES = {422--437},
YEAR = {1968}
}

@techreport{peltier_levy_vehel_multifractional_brownian_motion_definition_and_preliminary_results,
AUTHOR = {Peltier, R. F. and Lévy Véhel, J.},
TITLE = {{Multifractional Brownian Motion: Definition and Preliminary Results}},
type = {Research Report},
NUMBER = {RR-2645},
institution = {{INRIA}},
YEAR = {1995}
}

@article{Ayache_taqqu_multifractional_processes_with_random_exponent,
AUTHOR = {A. Ayache and M. S. Taqqu},
TITLE = {Multifractional processes with random exponent.},
JOURNAL = {Publicacions Matemàtiques},
VOLUME = {49},
NUMBER = {2},
PAGES = {459-486},
YEAR = {2005},
DOI = {10.5565/PUBLMAT_49205_11}
}

@article{ayache_esser_hamonier_a_new_multifractional_process_with_random_exponent,
AUTHOR = {A. Ayache and C. Esser and J. Hamonier},
TITLE ={A new Multifractional Process with Random Exponent},
JOURNAL = {Risk and Decision Analysis},
VOLUME = {7},
NUMBER = {1-2},
PAGES = {5-29},
YEAR = {2018},
DOI = {10.3233/RDA-180135}
}

@article{taqqu_wolpert_infinite_variance_self_similar_processes_subordinate_to_a_poisson_measure,
AUTHOR = {M. S. Taqqu and R. L. Wolpert},
TITLE = {Infinite Variance Self-Similar Processes Subordinate to a Poisson Measure},
JOURNAL = {Zeitschrift für Wahrscheinlichkeitstheorie und Verwandte Gebiete},
VOLUME = {6},
PAGES = {53-72},
YEAR = {1983},
DOI = {10.1007/BF00532163}}

@book{billingsley_convergence_of_probability_measures,
AUTHOR = {P. Billingsley},
TITLE = {Convergence of Probability Measures},
EDITION = {2},
PUBLISHER = {Wiley},
YEAR = {1999},
ISBN = {9780471197454},
DOI = {0.1002/9780470316962}
}

@book{kallenberg_foundations_of_modern_probability,
AUTHOR = {O. Kallenberg},
TITLE = {Foundations of Modern Probability},
series = {Probability Theory and Stochastic Modelling},
EDITION = {3},
PUBLISHER = {Springer Cham},
YEAR = {2021},
ISBN = {978-3-030-61870-4},
ISSN = {2199-3130},
DOI = {10.1007/978-3-030-61871-1}
}

@book{grafakos_classicall_fourier_analysis,
AUTHOR = {L. Grafakos},
TITLE = {Classical Fourier Analysis},
series = {Graduate Texts in Mathematics},
EDITION = {3},
PUBLISHER = {Springer New York},
YEAR = {2014},
ISBN = {978-1-4939-1193-6},
ISSN = {0072-5285},
DOI = {10.1007/978-1-4939-1194-3}
}

@book{doukhan_oppenheim_taqqu_theory_and_applications_of_long_range_dependence,
AUTHOR = {P. Doukhan and G. Oppenheim and M. S. Taqqu},
TITLE = {Theory and Applications of Long-Range Dependence},
EDITION = {1},
PUBLISHER = {Birkhäuser},
YEAR = {2002},
ISBN = {978-0-8176-4168-9}
}

@article{dangEstimationMultifractionalFunction2020,
    title = {Estimation of the multifractional function and the stability index of linear multifractional stable processes},
    volume = {24},
    copyright = {© EDP Sciences, SMAI 2020},
    issn = {1262-3318},
    url = {https://www.esaim-ps.org/articles/ps/abs/2020/01/ps180095/ps180095.html},
    doi = {10.1051/ps/2019012},
    language = {en},
    urldate = {2026-07-13},
    journal = {ESAIM: Probability and Statistics},
    publisher = {EDP Sciences},
    author = {Dang, Thi-To-Nhu},
    year = {2020},
    pages = {1--20},
}

@article{ayacheLinearMultifractionalStable2015,
    title = {Linear {Multifractional} {Stable} {Motion}: {Wavelet} {Estimation} of {H}(·) and $\alpha$ {Parameters}},
    volume = {55},
    issn = {1573-8825},
    shorttitle = {Linear {Multifractional} {Stable} {Motion}},
    url = {https://doi.org/10.1007/s10986-015-9272-1},
    doi = {10.1007/s10986-015-9272-1},
    language = {en},
    number = {2},
    urldate = {2026-07-13},
    journal = {Lithuanian Mathematical Journal},
    author = {Ayache, Antoine and Hamonier, Julien},
    month = apr,
    year = {2015},
    pages = {159--192},
}

\end{document}